\documentclass[12pt]{amsart}

\usepackage{amsmath}
\usepackage{amssymb}
\usepackage{amsthm}

\usepackage[all,cmtip]{xy}
\usepackage{hyperref}

\theoremstyle{plain}
\newtheorem{thm}{Theorem}[section]
\newtheorem{prop}[thm]{Proposition}
\newtheorem{lem}[thm]{Lemma}
\newtheorem{cor}[thm]{Corollary}
\newtheorem{conj}[thm]{Conjecture}
\newtheorem{prob}[thm]{Problem}

\theoremstyle{definition}
\newtheorem{defn}[thm]{Definition}
\newtheorem{eg}[thm]{Example}

\theoremstyle{remark}
\newtheorem{rem}[thm]{Remark}

\DeclareMathOperator{\Def}{Def}

\DeclareMathOperator{\Exc }{Exc}
\DeclareMathOperator{\Ext}{Ext}

\DeclareMathOperator{\Sing}{Sing}
\DeclareMathOperator{\Spec}{Spec}
\DeclareMathOperator{\depth}{depth}

\DeclareMathOperator{\Ker}{Ker}
\DeclareMathOperator{\Image}{Im}
\DeclareMathOperator{\Gal}{Gal}
\DeclareMathOperator{\mult}{mult}

\DeclareMathOperator{\id}{id}
\DeclareMathOperator{\Hom}{Hom}
\DeclareMathOperator{\wt}{wt}
\DeclareMathOperator{\ES}{ES}
\DeclareMathOperator{\Tot}{Tot}
\DeclareMathOperator{\pr}{pr}
\DeclareMathOperator{\Art}{Art}

\begin{document}

\title
[Deforming elephants]
{Deforming elephants of $\mathbb{Q}$-Fano threefolds}

\author{Taro Sano}
\address{Department of Mathematics, Faculty of Science, Kobe University, Kobe, 657, Japan}
\email{tarosano@math.kobe-u.ac.jp}

\maketitle

\begin{abstract}
We study deformations of a pair of a $\mathbb{Q}$-Fano $3$-fold $X$ with only terminal singularities and its elephant $D \in |{-}K_X|$. 
We prove that, if there exists $D \in |{-}K_X|$ with only isolated singularities, 
 the pair $(X,D)$ can be deformed to a pair of a $\mathbb{Q}$-Fano $3$-fold with only quotient singularities and a Du Val elephant. 
When there are only non-normal elephants, 
we reduce the existence problem of such a deformation to a local problem around the singular locus of the elephant.  
We also give several examples of $\mathbb{Q}$-Fano $3$-folds without Du Val elephants.  
\end{abstract}

\tableofcontents

\section{Introduction}
	In this paper, we consider algebraic varieties over the complex number field $\mathbb{C}$ unless otherwise stated. 
	
The main object of the study in this paper is an elephant of a $\mathbb{Q}$-Fano $3$-fold. 
A {\it $\mathbb{Q}$-Fano $3$-fold} is a normal projective $3$-fold with only terminal singularities whose anticanonical divisor is ample. 
For a normal variety $X$, a member of the anticanonical linear system $|{-}K_X|$ is called an {\it elephant} of $X$. 

The existence of a smooth elephant plays an important role in the classification of smooth Fano $3$-folds (cf. \cite{Iskovskikh1}, \cite{Iskovskikh2}). 
Shokurov and Reid proved that a Fano $3$-fold with only canonical Gorenstein singularities contains an elephant with only Du Val singularities. 
By using this result, Mukai (\cite{mukai}) classified the ``indecomposable'' Fano $3$-folds with 
canonical Gorenstein singularities. 

A $\mathbb{Q}$-Fano $3$-fold is one of the end products of the Minimal Model Program and it has non-Gorenstein singularities. 
There are much more families of $\mathbb{Q}$-Fano $3$-folds than Gorenstein ones and their classification is not completed.  

In the non-Gorenstein case, a $\mathbb{Q}$-Fano $3$-fold $X$ may have empty anticanonical system or  
 have only non-Du Val elephants even if $|{-}K_X| \neq \emptyset$. Actually, such examples are given in \cite{Iano-Fletcher}, \cite[4.8.3]{ABR} (See also Section \ref{examplesection}).  
Moreover, although $3$-fold terminal singularities are already classified (\cite{YPG}, \cite{mori}), there are still many of them and it is complicated to treat them. 

Locally, a $3$-fold terminal singularity can be deformed to a variety with quotient singularities (\cite[(6.4)]{YPG}).
It is easier to treat $\mathbb{Q}$-Fano $3$-folds with only quotient singularities and with Du Val elephants.  
For example, Takagi (\cite{TakagiDuVal}) established a bound of $h^0(X, -K_X)$ of ``primary'' $\mathbb{Q}$-Fano $3$-folds $X$ with these conditions 
and classified such $\mathbb{Q}$-Fano $3$-folds with   
$h^0(X,-K_X) =9, 10$. 

There are several attempts to reduce to such treatable situations. 
Alexeev proved that, if $|{-}K_X |$ sufficiently moves, a $\mathbb{Q}$-Fano $3$-fold with only non-Du Val elephants 
is birational to one with a Du Val elephant (\cite[Theorem 4.3]{AlexeevElephant},\cite[Theorem 11.1.8]{IPencyclopedia}).   
As a deformation-theoretic approach, there is the following conjecture by Alt{\i}nok--Brown--Reid \cite[4.8.3]{ABR}. 

\begin{conj}\label{simulqsmconj}
Let $X$ be a $\mathbb{Q}$-Fano $3$-fold. Then the following hold.  
\begin{enumerate}
\item[(i)] There exists a deformation $f \colon \mathcal{X} \rightarrow \Delta^1$ of $X$ over a unit disc such that the general fiber $\mathcal{X}_t$ 
is quasi-smooth for $t \neq 0$, that is, it has only quotient singularities. (Such a deformation of $X$ is called a {\bf $\mathbb{Q}$-smoothing} of $X$.) 
\item[(ii)] Assume that $|{-} K_X|$ contains an element $D$. Then there exists a deformation $f \colon (\mathcal{X}, \mathcal{D}) \rightarrow \Delta^1$ 
of the pair $(X,D)$ 
such that $\mathcal{X}_t$ is quasi-smooth and $\mathcal{D}_t \in |{-}K_{\mathcal{X}_t}|$ has Du Val singularities only on the singularities of $\mathcal{X}_t$
for $t \neq 0$. (Such a deformation of $(X,D)$ is called a {\bf simultaneous $\mathbb{Q}$-smoothing} of a pair $(X,D)$. 
See also Definition \ref{simultaneousQsmoothingdefn}.) 
\end{enumerate}
\end{conj} 

Conjecture \ref{simulqsmconj} (i) is solved in most of the cases as follows. 

\begin{thm}\label{qsmqfanothmintro}(\cite[Theorem 1.1]{Sanonote})
	A $\mathbb{Q}$-Fano $3$-fold can be deformed to one with only quotient singularities and 
	  $A_{1,2}/4$-singularities.  
	\end{thm} 
  Here an {\it $A_{1,2}/4$-singularity} means a singularity locally isomorphic to 
a ``hyper-quotient'' singularity 
 $(x^2+y^2+z^3+u^2=0)/\mathbb{Z}_4(1,3,2,1)$. (See Section \ref{weightedblupsection} for the notation.)

The main issue in this paper is Conjecture \ref{simulqsmconj} (ii). 
A typical example of a simultaneous $\mathbb{Q}$-smoothing can be given for a quasi-smooth $\mathbb{Q}$-Fano weighted hypersurface in Iano-Fletcher's list \cite{Iano-Fletcher}. 
If we take some special equation, it does not have a Du Val elephant. 
However, a general member of the family contains a Du Val K3 surface as its elephant. 
(We explain this phenomenon in Example \ref{egwhsnonlcelephant}.)  

The following is the main result of this paper. 

\begin{thm}\label{simulqsmisolthmintro}{\rm (= Theorem \ref{maintheorem})} 
	Let $X$ be a $\mathbb{Q}$-Fano $3$-fold. Assume that there exists $D \in |{-}K_X|$ with only isolated singularities. 
	
	Then there exists a simultaneous $\mathbb{Q}$-smoothing of $(X,D)$. 
	
	In particular, $X$ has a $\mathbb{Q}$-smoothing.  
	\end{thm}

The statement of Theorem \ref{simulqsmisolthmintro} is not empty since there is an example of a $\mathbb{Q}$-Fano $3$-fold with only terminal quotient singularities and
 with only non Du Val elephants (Example \ref{egwhsnonlcelephant}). 
Also note that we do not need the assumption on terminal singularities on $X$.  
 
 \vspace{5mm}
 
When $X$ has only non-normal elephants, the problem becomes more subtle. 
 There is an example of a klt  $\mathbb{Q}$-Fano $3$-fold with only isolated cyclic quotient singularities such that 
its small deformations have only non-normal elephants (See Example \ref{kltnonnormalexample}). 
However, we can at least reduce the problem to a local problem as follows. 

\begin{thm}\label{nonisolsimulqsmintro}
	Let $X$ be a $\mathbb{Q}$-Fano $3$-fold. Assume that $D \in |{-}K_X|$ has a reduced element $D$. 
	Let  $C:= \Sing D$ be the singular locus of $D$, $U_C \subset X$ an analytic open neighborhood of $C$ and $D_C:= D \cap U_C$. 
	Assume that the pair $(U_C, D_C)$ has a simultaneous $\mathbb{Q}$-smoothing. 
	
	Then the pair $(X,D)$ also has a simultaneous $\mathbb{Q}$-smoothing. 
	\end{thm}

\subsection{Strategy of the proof} 
To prove Theorem \ref{simulqsmisolthmintro}, we use a coboundary map of a local cohomology group associated to a certain resolution of the pair $(X,D)$. 
Namikawa--Steenbrink also used some coboundary map to prove the  
smoothability of Calabi--Yau $3$-folds with terminal singularities (\cite[Section 1]{NamSt}). 
While they can use arbitrary log resolution of singularities in their definition of the coboundary map, 
we need to choose a special resolution carefully. 
We shall give a sketch of the proof of Theorem \ref{simulqsmisolthmintro} in the following.

Let $\Sing D := \{p_1, \ldots, p_l \}$, $U_i \subset X$ a Stein neighborhood of $p_i$ for $i=1, \ldots, l$ and $D_i:= D \cap U_i$. 
Let $T^1_{(X,D)}, T^1_{(U_i, D_i)}$ be the sets of first order deformations of the pair $(X,D)$ and $(U_i,D_i)$ respectively. 
Since deformations of the pair $(X,D)$ are unobstructed (\cite[Theorem 2.17]{Sano}), it is enough to find an element $\eta \in T^1_{(X,D)}$ 
which deforms singularities of $D_i$. 
By Theorem \ref{qsmqfanothmintro}, we can assume that $U_i$ is locally isomorphic to $\mathbb{C}^3/ \mathbb{Z}_r(1,a,r-a)$ or 
$(x^2+y^2+z^3+u^2=0)/\mathbb{Z}_4(1,3,2,1)$.  Since $U_i$ contains a Du Val elephant (cf.\ \cite{YPG}), 
there exists $\eta_i \in T^1_{(U_i, D_i)}$ which induces a simultaneous $\mathbb{Q}$-smoothing of $(U_i, D_i)$. 
We study the restriction homomorphism $\oplus p_{U_i} \colon T^1_{(X,D)} \rightarrow \oplus_{i=1}^l T^1_{(U_i,D_i)}$ and 
want to lift a local deformation $\eta_i \in T^1_{(U_i, D_i)}$. 
There exists an exact sequence 
\[
T^1_{(X,D)} \stackrel{\oplus p_{U_i}}{\rightarrow} \oplus_{i=1}^l T^1_{(U_i,D_i)} \rightarrow H^2(X, \Theta_X(- \log D)),   
\]
where $\Theta_X(- \log D)$ is the sheaf of tangent vectors which vanish along $D$. 
One direct approach is to try to prove $H^2(X, \Theta_{X}(- \log D)) =0$. 
However, this strategy does not work well. 
Thus we should study the map $\oplus p_{U_i}$ more precisely. 

For this purpose, we use some local cohomology groups supported on the exceptional divisor of a suitable resolution $\mu_i \colon \tilde{U_i} \rightarrow U_i$ 
of the pair $(U_i, D_i)$ for $i=1 \ldots, l$.  
We use the commutative diagram of the form 
\[
\xymatrix{
T^1_{(X,D)} \ar[r]^{\oplus p_{U_i}} \ar[rd]_{\oplus \psi_i}  
& \oplus_{i=1}^l T^1_{(U_i,D_i)} \ar[d]^{\oplus \phi_i} \\ 
 & \oplus_{i=1}^l H^2_{E_i}(\tilde{U}_i, \Omega^2_{\tilde{U}_i} (\log \tilde{D}_i +E_i)), 
}
\]
where $\tilde{D}_i \subset \tilde{U}_i$ is the strict transform of $D_i$, $E_i:= \mu_i^{-1}(p_i)$ and $\phi_i$ is 
the coboundary map. 

One of the key points of the proof is to show that the coboundary map $\phi_i$ does not vanish. 
Actually, in Lemma \ref{lem:logkerm^2} (ii), when $U_i$ has only quotient singularity, 
we show that 
\begin{equation}\label{m2relintro}
\Ker \phi_i \subset \mathfrak{m}^2  T^1_{(U_i,D_i)},
\end{equation} 
where $\mathfrak{m}^2  T^1_{(U_i,D_i)}$ is the set of deformations induced by functions of order $2$ or higher (See (\ref{m2definition}) for the definition). 
In order to show this, we should carefully choose a resolution $\mu_i \colon \tilde{U}_i \rightarrow U_i$. 
We first choose a suitable weighted blow-up $\mu_{i,1} \colon U_{i,1} \rightarrow U_i$ such that 
$K_{U_{i,1}}+ (\mu_{i,1}^{-1})_* (D_i) - \mu_{i,1}^* (K_{U_i} +D_i)$ 
has negative coefficient (Lemma \ref{nonGorwtlem} and Lemma \ref{A12-wtdblowup}).    
Next we construct a suitable resolution $\mu_{i,12} \colon U_{i,2} \rightarrow U_{i,1}$ of the pair 
$(U_{i,1}, \mu_{i,1}^{-1}(D_i))$ (Lemma \ref{U_2lem}). 
By these careful choices, we can prove the statement (\ref{m2relintro}) in Lemma \ref{lem:logkerm^2}. 
This subtlety of choosing a suitable resolution does not appear in the previous approach of finding 
a global smoothing as in \cite{NamSt}, for example. Thus this issue is a new feature of our method.  

We can also show that $\psi_i$ is surjective since $D$ is ample and $X \setminus D$ is affine. 
Here we need the Fano assumption. 

By these two statements, we can do diagram-chasing in the above diagram to find a 
simultaneous $\mathbb{Q}$-smoothing direction $\eta \in T^1_{(X,D)}$. This is a sketch of the proof. 

The framework of the proof of Theorem \ref{nonisolsimulqsmintro} is similar.  

\section{Preliminaries on deformations of a pair}

\subsection{Deformation of a morphism of algebraic schemes}

In this paper, we often use the following notion of a functor of deformations of a closed immersion of 
algebraic schemes. 

\begin{defn}\label{pairdefschdefn}(cf. \cite[3.4.1]{Sernesi})
Let $f \colon D \hookrightarrow X$ be a closed immersion of algebraic schemes over an algebraically closed field $k$ and $S$ an algebraic scheme over $k$ with 
a closed point $s \in S$. 
A {\it deformation} of a pair $(X,D)$ over $S$ is a data $(F, i_X, i_D)$ in the cartesian diagram 
\begin{equation}\label{deformmapdiagram}
\xymatrix{
D \ar@{^{(}->}[r]^{i_D} \ar[d]^{f} & \mathcal{D} \ar[d]^{F} \\
X \ar@{^{(}->}[r]^{i_X}  \ar[d] & \mathcal{X} \ar[d]^{\Psi} \\ 
\{s \} \ar@{^{(}->}[r] & S, 
}
\end{equation}
where $\Psi$ and $\Psi \circ F$ are flat and $i_D, i_X$ are closed immersions. 
Two deformations $(F, i_D, i_X)$ and $(F', i_D', i_X')$ of $(X,D)$ over $S$ are said to be 
{\it equivalent} if there exist isomorphisms $\alpha \colon \mathcal{X} \rightarrow \mathcal{X}'$ and 
$\beta \colon \mathcal{D} \rightarrow \mathcal{D}'$ over $S$ 
which commutes the following diagram; 
\[
\xymatrix{
D \ar@{^{(}->}[r]^{i_D} \ar@{^{(}->}[rd]^{i'_{D}} & \mathcal{D} \ar[r] \ar[d]^{\beta} & \mathcal{X} \ar[d]^{\alpha} & 
X \ar@{^{(}->}[l]^{i_X} \ar@{^{(}->}[ld]^{i_X'} \\
 & \mathcal{D}' \ar[r] & \mathcal{X}'. &   
}
\] 

Let $\Art_k$ be the category of Artin local $k$-algebras with residue field $k$. 
We define the functor $\Def_{(X,D)} \colon \Art_k \rightarrow (Sets)$ by setting 
\begin{equation}\label{pairfunctordef}
\Def_{(X,D)}(A):= \{ (F, i_D, i_X): \text{deformation of } (X,D) \text{ over } \Spec A \}/({\rm equiv}),    
\end{equation}
where $({\rm equiv})$ means the equivalence introduced in the above. 

Similarly, we can define the deformation functor $\Def_X \colon \Art_k \rightarrow (Sets)$ of an algebraic scheme $X$. Actually, we have $\Def_X= \Def_{(X,\emptyset)}$. 
\end{defn}

In this paper, we just treat the cases when $D$ is a divisor. 
Next, we introduce the notion of a deformation of a pair of a variety and several effective Cartier divisors. 

\begin{defn}\label{defn-pair}
Let $X$ be an algebraic variety and $D_j$ for $j \in J$ a finite number of effective Cartier divisors. Set $D := \sum_{j \in J} D_j$. 
We can define a functor $\Def_{(X,D)}^J \colon \Art_k \rightarrow (Sets)$ by setting $\Def_{(X,D)}^J(A)$ to be the equivalence classes of 
deformations of a closed immersion $i \colon D \hookrightarrow X$ induced by deformations of each divisor $D_j \hookrightarrow X$
 for $A \in \Art_k$. 

We skip the script $J$ when $D = \sum_{j \in J} D_j$ is the decomposition into irreducible components and there is no confusion. 
In this paper, we only treat deformations of a divisor coming from deformations of its irreducible components.   

Let $A_1:= k[t]/(t^2)$. 
In this setting, we write $T^1_{(X,D)}:= \Def_{(X,D)}(A_1)$ and $T^1_{X}:= \Def_X(A_1)$ for the sets of first order deformations of the pair $(X,D)$ and $X$, 
respectively.  
\end{defn}

\begin{rem}
Let $X$ be a smooth variety and $D= \sum D_j$ a SNC divisor on $X$.
Then we have the well-known isomorphism 
\[
T^1_{(X,D)} \simeq H^1(X, \Theta_X (- \log D)), 
\]
where $\Theta_{X}(- \log D)$ is the sheaf of tangent vectors on $X$ vanish along $D$ (cf. \cite[Proposition 3.4.17]{Sernesi}). 
\end{rem}

\begin{rem}
If $X$ is smooth and $D= \sum_{j \in J} D_j$ is a SNC divisor, $\Def_{(X,D)}^J(A)$ parametrizes locally trivial deformations and
 does not include an element which induces a smoothing of $D$. 
\end{rem}

\subsection{Preliminaries on weighted blow-up}\label{weightedblupsection}
We prepare several properties of the weighted blow-up. 
We refer  \cite[Section 2]{Hayakawa2} for more details. 

Let $v := \frac{1}{r} (a_1, \ldots, a_n) \in \frac{1}{r} \mathbb{Z}^n$, $N:= \mathbb{Z}^n + \mathbb{Z} v$ a lattice and 
$M:= \Hom (N, \mathbb{Z})$. 
Let $e_1:= (1,0, \ldots,0), \ldots, e_n:= (0,\ldots,0,1)$ be a basis of $N_{\mathbb{R}} := N \otimes \mathbb{R}$ and  
 $\sigma := \mathbb{R}_{\ge 0}^{n} \subset \mathbb{R}^n$  the cone generated by $e_1, \ldots, e_n$. 
 Let $U:= \Spec \mathbb{C}[\sigma^{\vee} \cap M]$ be the associated toric variety. 
 We know that $U \simeq \mathbb{C}^n/ \mathbb{Z}_r(a_1,\ldots ,a_n)$, where the R.H.S.\ is the quotient of $\mathbb{C}^n$ 
 by the $\mathbb{Z}_r$-action 
 \[
 (x_1, \ldots, x_n) \mapsto (\zeta_r^{a_1} x_1, \ldots, \zeta_r^{a_n} x_n), 
 \]
where $x_1, \ldots, x_n$ are the coordinates on $\mathbb{C}^n$ and $\zeta_r$ is a primitive $r$-th root of unity. 
In this case, we say that $0 \in \mathbb{C}^n/ \mathbb{Z}_r$ is a $1/r(a_1, \ldots, a_n)$-singularity. 

Let $v_1:= \frac{1}{r}(b_1, \ldots, b_n) \in N$ be a primitive vector such that $b_i >0$ for all $i$. 
Let $\Sigma_1$ be a fan which is formed by the cones $\sigma_i$ generated by $\{e_1, \ldots, e_{i-1}, v_1, e_{i+1}, \ldots, e_n \}$
for $i=1, \ldots, n$. 
Let $U_1$ be the toric variety associated to the fan $\Sigma_1$. 
Let $\mu_1 \colon U_1 \rightarrow U$ be the toric morphism associated to the subdivision. 
It is a projective birational morphism with an exceptional divisor $E_1:= \mu_1^{-1}(0)$. 
We call $\mu_1$ the {\it weighted blow-up } with weights $v_1$.  

Let $f:= \sum f_{i_1, \ldots, i_n} x_1^{i_1} \cdots x_n^{i_n} \in \mathbb{C}[x_1, \ldots,x_n]$ be the 
$\mathbb{Z}_r$-semi-invariant polynomial with respect to the $\mathbb{Z}_r$-action on $\mathbb{C}^n$. 
Let 
\[
\wt_{v_1}(f):= \min \{ \sum_{j=1}^n \frac{b_j i_j}{r} \mid f_{i_1, \ldots, i_n} \neq 0 \}
\] 
be {\it the $v_1$-weight } of $f$. 
Let $D_f:= (f=0)/ \mathbb{Z}_r \subset U$ be the divisor determined by $f$ and 
$D_{f,1} \subset U_1$ the strict transform of $D_f$. 
Then we have the following; 
\begin{equation}\label{discrepancywtblup}
	K_{U_1} = \mu_1^* K_U + \frac{1}{r}(\sum_{i=1}^n b_i -r) E_1, 
	\end{equation} 
\begin{equation}\label{stricttrfwtblup}
	D_{f,1} = \mu_1^* D_f - \wt_{v_1}(f) E_1.  
	\end{equation} 

 Let $U_{1,i} \subset U_1$ be the affine open subset which corresponds to the cone $\sigma_i$. 
 Then we have $U_1 = \bigcup_{i=1}^n U_{1,i}$ and 
 \[
 U_{1,i} \simeq \mathbb{C}^n/ \mathbb{Z}_{a_i}(-a_1, \ldots, \overset{i{\rm -th}}{r}, \ldots, - a_n). 
 \]
Moreover the morphism  $ \mu_1|_{U_{1,i}} \colon U_{1,i} \rightarrow U$ is described by 
\[
(x_1, \ldots, x_n) \mapsto (x_1 x_i^{a_1/r}, \ldots, x_i^{a_i/r}, \ldots, x_n x_i^{a_n/r}). 
\]

\subsection{Deformations of a divisor in a terminal $3$-fold}
We first define discrepancies of a log pair. 

\begin{defn}
	Let $U$ be a normal variety and $D$ its divisor such that $K_U +D $ is $\mathbb{Q}$-Cartier, that is, 
	$m(K_U+D)$ is a Cartier divisor for some positive integer $m$. 
	Let $\mu \colon \tilde{U} \rightarrow U$ be a proper birational morphism from another normal variety 
	and $E_1, \ldots, E_l$ its exceptional divisors. 
	Let $\tilde{D} \subset \tilde{U}$ be the strict transform of $D$. 
	
	We define a rational number $a(E_i, U, D)$ as the number such that   
	\[
	m(K_{\tilde{U}} +\tilde{D}) = \mu^*(m(K_U +D)) + \sum_{i=1}^l m a(E_i, U, D) E_i. 
	\]
	We call $a(E_i, U, D)$ the {\it discrepancy} of $E_i$ with respect to the pair $(U,D)$. 
	\end{defn}

Let $U$ be a Stein neighborhood of a $3$-fold terminal singularity of Gorenstein index $r$ and 
$D$ a $\mathbb{Q}$-Cartier divisor on $U$. 
We have the index one cover $\pi_U \colon V:= \Spec \oplus_{j=0}^{r-1} \mathcal{O}_U(jK_U) \rightarrow U$ 
determined by an isomorphism $\mathcal{O}_U(rK_U) \simeq \mathcal{O}_U$. 
Let $G:= \Gal (V/U)\simeq \mathbb{Z}_r$ be the Galois group of $\pi_U$. 
 This induces a $G$-action 
on the pair $(V, \Delta)$, where $\Delta:= \pi_U^{-1}(D)$.  
We can define  functors of $G$-equivariant deformations of $(V, \Delta)$ as follows. 

\begin{defn}
	Let $\Art_{\mathbb{C}}$ be the category of local Artinian $\mathbb{C}$-algebras with residue field $\mathbb{C}$. 
Let $\Def^G_{(V,\Delta)} \colon \Art_{\mathbb{C}} \rightarrow (Sets)$ be a functor such that, for $A \in \Art_{\mathbb{C}}$, a set   
$\Def_{(V,\Delta)}^G(A) \subset \Def_{(V, \Delta)}(A)$ is the set of deformations $(\mathcal{V}, \mathbf{\Delta})$ of $(V, \Delta)$ over $A$ 
with a $G$-action which is compatible with the $G$-action on $(V, \Delta)$.   

We can also define the functor $\Def^G_{V} \colon \Art_{\mathbb{C}} \rightarrow (Sets)$ of $G$-equivariant deformations of $V$ similarly. 
\end{defn}

\begin{prop}\label{Gequivdeffuncprop}(\cite[Proposition 2.15]{Sano},\cite[Proposition 3.1]{Namtop})
We have isomorphisms of functors 
\begin{equation}\label{Gequivfunctisom}
 \Def^G_{(V,\Delta)} \simeq \Def_{(U,D)}, \ \ \ \  \Def^G_{V} \simeq \Def_U. 
\end{equation} 
Moreover, these functors are unobstructed and the forgetful homomorphism 
$\Def_{(U,D)} \rightarrow \Def_U$ is a smooth morphism of functors.  
\end{prop}

This proposition implies the following. 

\begin{prop}\label{T^1invprop}
	Let $U,D, \pi_U \colon V \rightarrow U, \Delta$ as above. 
	Then we have 
	\[
	T^1_{(U,D)} \simeq (T^1_{(V,\Delta)})^{G}, \ \ \ \  T^1_U \simeq (T^1_V)^{G}. 
	\]
	\end{prop}
 
 We check these isomorphisms in the following examples.

\begin{eg}
	Let $U:= \mathbb{C}^3/ \mathbb{Z}_2(1,1,1)$ and $D:= (x^3 +y^3 +z^3=0)/ \mathbb{Z}_2 \subset U$ its divisor. 
	In this case, we can write $V= \mathbb{C}^3$ and $\Delta = (x^3 +y^3 +z^3=0) \subset V$. 
	For 
	$f \in \mathcal{O}_{\mathbb{C}^3,0}$ such that $g \cdot f = -f$, 
	let $\eta_f \in T^1_{(U,D)}$ be the deformation $(\mathcal{U}, \mathcal{D}_f)$ of 
	$(U,D)$ over $A_1= \mathbb{C}[t]/(t^2)$ defined as follows; 
	\[
	\mathcal{U}:= U \times \Spec A_1,
	\] 
	\[
	\mathcal{D}_f:= (x^3+y^3+z^3+t f=0)/\mathbb{Z}_2 \subset \mathcal{U}.
	\]
	Then we have 
	\begin{equation*}
	T^1_{(U,D)} \simeq (T^1_{(V,\Delta)})^{\mathbb{Z}_2} \simeq (\mathcal{O}_{\mathbb{C}^3,0}/(x^2, y^2, z^2))^{[-1]} 
	\simeq \mathbb{C}\eta_x \oplus \mathbb{C} \eta_y \oplus \mathbb{C}\eta_z \simeq \mathbb{C}^3,   
	\end{equation*}
	where $(\mathcal{O}_{\mathbb{C}^3,0}/(x^2, y^2, z^2))^{[-1]}:= \{ f \in \mathcal{O}_{\mathbb{C}^3,0}/(x^2, y^2, z^2) \mid g \cdot f = -f \}$. 
	\end{eg}

We often use push-forward of an exact sequence by an open immersion. 

\begin{prop}\label{openpushforwardexact}
	Let $X$ be an algebraic scheme and $Z \subset X$ a closed subset. 
	Let $\iota \colon X \setminus Z \hookrightarrow X$ be the open immersion and 
	\[
	0 \rightarrow \mathcal{F} \rightarrow \mathcal{G} \rightarrow \mathcal{H} \rightarrow 0
	\]
	an exact sequence of coherent sheaves on $U:= X \setminus Z$. 
Assume that $\depth_p \iota_* \mathcal{F} \ge 3$ for all scheme-theoretic points $p \in Z$.  

Then we obtain an exact sequence 
\[
0 \rightarrow \iota_* \mathcal{F} \rightarrow \iota_* \mathcal{G} \rightarrow \iota_* \mathcal{H} \rightarrow 0. 
\]	
	\end{prop}

\begin{proof}
	We have $R^1 \iota_* \mathcal{F} =0$ by the condition on the depth of $\iota_* \mathcal{F}$. 
	This implies the required surjectivity.
	\end{proof}

Proposition \ref{openpushforwardexact} immediately implies the following lemma on a restriction homomorphism of extension groups. 

\begin{lem}\label{codim3Omega1**}
	Let $X, Z, U$ be as in Proposition \ref{openpushforwardexact}. Let $\mathcal{F}$ be a reflexive coherent sheaf on $X$.
	Assume that $\depth_p \mathcal{O}_{X,p} \ge 3$ for all scheme theoretic points $p \in Z$. 
	
	 Then we have 
	\[
	\Ext^1(\mathcal{F}, \mathcal{O}_X ) \simeq \Ext^1(\mathcal{F}|_U, \mathcal{O}_U). 
	\]
	\end{lem}

We repeatedly use the following lemma which is also a consequence of Proposition \ref{openpushforwardexact}. 

\begin{lem}\label{codim3pair}(\cite[Lemma 4.3]{Sano}) 
Let $X$ be a $3$-fold with only terminal singularities and $D$ a $\mathbb{Q}$-Cartier divisor on $X$. 
Let $Z \subset X$ be a $0$-dimensional subset. Let $\iota \colon U:= X \setminus Z \hookrightarrow X$ be an open immersion.
Set $D_U := D \cap U$.  
 
Then the restriction homomorphism $\iota^* \colon T^1_{(X,D)} \rightarrow T^1_{(U,D_U)}$ is an isomorphism.  
\end{lem}

\subsection{Additional lemmas}
We need the following lemma essentially due to Professor Angelo Vistoli. 

\begin{lem}\label{vistoli}
Let $f \in \mathbb{C}[x,y,z]$ be a polynomial which defines a reduced divisor $0 \in D:= (f=0) \subset \mathbb{C}^3$ 
and $\Gamma:= \Sing D$. 
Assume that a polynomial $g \in \mathbb{C}[x,y,z]$ defines a smoothing $\mathcal{D}:= (f+tg =0) \subset \mathbb{C}^3 \times \mathbb{C}$ 
 of $D$ over the affine line $\mathbb{C}$. 
Let $h \in \mathbb{C}[x,y,z]$ be a polynomial such that  $\mult_p h \ge 2$ for $p \in \Gamma$.
Then $\mathcal{D}':= (f+t(h+g)=0)$ is also a smoothing of $D$.  
 \end{lem}

\begin{proof}
	Note that $\mult_p g \le 1$ for $p \in \Gamma$ since $(f+tg=0)$ is a smoothing. 
Consider the linear system 
\[
\{C_{[s:t]} := (sf+t(h+g) =0) \subset \mathbb{C}^3 \mid [s:t] \in \mathbb{P}^1 \}. 
\]
By Bertini's theorem, the divisor $C_{[s:t]}$ is smooth away from the base points of the linear system, for all but finitely many values of $[s:t]$.  
If $p \in \mathbb{C}^3$ is a base point of the linear system, then either $p \in \Gamma$, in which case $C_{[0:1]}$ is smooth at $p$, 
or is not, and in this case, $C_{[1:0]}$ is smooth at $p$. 
Since being smooth at a base point is an open condition, we have that $C_t$ is smooth at all points of $\mathbb{C}^3$ for all but finitely many values of $t$.
\end{proof}

\vspace{2mm}

We also use the following lemma on the vanishing of a cohomology group on a toric variety 
which is a consequence of the vanishing theorem due to Fujino (\cite{Fujino}). 

\begin{lem}\label{QFujinotoric}
	Let $U$ be an affine toric variety and 
 $\pi \colon V \rightarrow U$ a projective toric morphism of toric varieties. 
 Let $V' \subset V$ be the smooth locus of $V$ and 
 $\iota \colon V' \hookrightarrow V$ the open immersion.  
  Let $D$ be a $\pi$-ample $\mathbb{Q}$-Cartier divisor on $V$ and $D':= D|_{V'}$ 
  its restriction on $V'$. 
 
 Then we have 
 \[
H^i(V, \iota_* (\Omega^j_{V'}(D'))) =0 
\]
for $i>0$ and $j \ge 0$. 
	\end{lem}
	
	\begin{proof} 
		
		By the Serre vanishing theorem, 
		we can take  a sufficiently large integer $l$ such that $lD$ is $\pi$-very ample and 
		$H^i(V, \iota_*\Omega^j_{V'} (lD))=0$. 
		
		Let $F:=F_l \colon V \rightarrow V$ be the $l$-times multiplication map as introduced in \cite[2.1]{Fujino}. 
		Note that our symbol $V'$ is different from that in \cite{Fujino}. 
		Let $F':= F|_{V'} \colon V' \rightarrow V'$ be the multiplication map on $V'$. 
		Note that we have a split injection $\Omega^j_{V'} \hookrightarrow F'_* \Omega^j_{V'}$ (\cite[2.6]{Fujino}). 
		Since we have $(F')^* D' = lD'$, we obtain  
		\begin{equation}
		F_* \iota_*\Omega^j_{V'}((F')^*D') )
		= \iota_* F'_* \Omega^j_{V'}((F')^* D')) 
		 \simeq \iota_* (F'_* \Omega^j_{V'})(D') \hookleftarrow \iota_* (\Omega^j_{V'}(D')). 
		\end{equation}
		This implies that 
		\[
		H^i(V, \iota_* (\Omega^j_{V'}(D'))) \hookrightarrow H^i(V, F_* (\iota_*(\Omega^j_{V'}((F')^*D') ))) \simeq  H^i(V, \iota_* \Omega^j_{V'}(lD))=0. 
		\]
		Thus we obtain $H^i(V, \iota_* (\Omega^j_{V'}(D'))) =0$. 
		We finish the proof of Lemma \ref{QFujinotoric}.  
		\end{proof}

\subsection{Blow-down morphism of deformations}

	Let $X$ be an algebraic variety and $\tilde{X} \rightarrow X$ its resolution of singularities. 
	Suppose we have a deformation $\tilde{\mathcal{X}} \rightarrow \Spec A$ over an Artin ring $A$. 
	If $X$ has only rational singularities, 
	we can ``blow-down'' the deformation $\tilde{\mathcal{X}}$ to a deformation of $X$.   

	We need the following proposition in general setting. 

	\begin{prop}\label{bldownlem}(\cite[Section 0]{wahl})
	Let $X$ be an algebraic scheme over an algebraically closed field $k$ and $A \in \Art_k$. 
	Let $\mathcal{X} \rightarrow \Spec A$ be a deformation of $X$ and 
	$\mathcal{F}$ a quasi-coherent sheaf on $\mathcal{X}$, flat over $A$, inducing 
	$F:= \mathcal{F} \otimes_A k$ on $X$. 
	Let 
	\[
	\phi^0 \colon H^0(\mathcal{X}, \mathcal{F}) \otimes_A k \rightarrow H^0(X, F)
\]
be the natural restriction map. 

	If $H^1(X, F) =0$, then $\phi^0$ is an isomorphism and $H^0(\mathcal{X}, \mathcal{F})$ is $A$-flat. 
	\end{prop}

	Proposition \ref{bldownlem} implies the following. 

	\begin{cor}\label{blowdowncor} 
	Let $f \colon X \rightarrow Y$ be a proper birational morphism of integral normal $k$-schemes. 
	Assume that $R^1 f_* \mathcal{O}_Y =0$. 

	Then there exists a morphism of functors 
	\[
	f_* \colon \Def_X \rightarrow \Def_Y 
	\]
	defined as follows: For a deformation $\mathcal{X} \rightarrow \Spec A$ of $X$ over $A \in \Art_k$, we define 
	its image by $f_*$ as the scheme $\mathcal{Y} = (Y, f_* \mathcal{O}_{\mathcal{X}})$.  

	We call this transformation the {\bf blow-down} morphism. 
	\end{cor}

	For a surface with non-rational singularities, Wahl considered ``equisingularity'' of deformations 
	via the blow-down transformation. Although the blow-down transformation is not always possible, 
	we can still consider the ``equisingular deformation functor'' as follows. 

	\begin{defn}
		Let $ U:= \Spec R$ be an affine normal surface over an algebraically closed field $k$ with a singularity at $p \in U$ 
		and $f \colon X \rightarrow U$ a resolution of a singularity such that $f^{-1}(p)$ has SNC support. 
		Wahl (\cite[(2.4)]{wahl}) defined an equisingular deformation of the resolution of a singularity as 
		a deformation of $(X,E)$ whose ``blow-down'' can be defined.  
		 More precisely, he defined a functor $\ES_X \colon \Art_{k} \rightarrow (Sets)$ by setting  
		 \[
		 \ES_X (A):= \{ 
		 (\mathcal{X}, \mathcal{E}) \in \Def_{(X,E)}(A) \mid H^0(\mathcal{X}, \mathcal{O}_{\mathcal{X}}): A\text{-flat}. 
		 \}, 
		 \]
		 where $\Art_k$ is a category of local Artin $k$-algebras with residue field $k$. 
		 There exists a natural transformation $f_* \colon \ES_X \rightarrow \Def_{U}$ and this induces a linear map 
		 $f_*(A_1) \colon \ES_X (A_1) \rightarrow \Def_U(A_1)$ on the tangent spaces. 
	\end{defn}

		 Equisingular deformation should preserve some properties of a singularity. 
		 For example, it is known that equisingular deformations of an isolated $2$-dimensional hypersurface singularity 
		 do not change the Milnor number (\cite{wahl}). 
		 In particular, smoothings of a hypersurface singularity can not be equisingular. 
		 However, the situation is a bit different in higher codimension case.  
		 Although a singularity has high multiplicity in general, an equisingular deformation may be induced by an equation of multiplicity one. 
		 This phenomenon does not happen in the hypersurface case as shown in Lemma \ref{vistoli}. 
		 In the following, we exhibit such an example due to Wahl (\cite{wahlpersonal}) of a deformation of an isolated complete intersection singularity (ICIS for short). 

	\begin{eg}	 
		 Let $U:= (xy-z^2 = x^4 +y^4 +w^2=0) \subset \mathbb{C}^4$ be an ICIS and 
		 $\mathcal{U}:= (xy - z^2 +tw = x^4 +y^4 +w^2=0) \subset \mathbb{C}^4 \times \mathbb{C}$ a deformation of $U$, 
		 where $x,y,z,w$ are coordinates on $\mathbb{C}^4$ and $t$ is a deformation parameter of $\mathbb{C}$. 
		For any value of $t$, the singularity $\mathcal{U}_t$ is a cone $(C_t, K_{C_t})$ for a smooth curve $C_t$ of genus $3$ and 
		its canonical bundle, that is, $\mathcal{U}_t \simeq \Spec \oplus_{k=0}^{\infty} H^0(C_t, kK_{C_t})$. 
		We see that $C_0 \simeq (xy-z^2 =x^4 +y^4 +w^2 =0) \subset \mathbb{P}(1,1,1,2)$ is a hyperelliptic curve 
		and $C_t$ for $t \neq 0$ is a smooth quartic curve in $\mathbb{P}^2$.  
		The singularity has a resolution 
		\[
		f_t \colon \Tot (\mathcal{O}_{C_t}(K_{C_t}) := \Spec \oplus_{k=0}^{\infty} \mathcal{O}_{C_t}(k K_{C_t}) \rightarrow \mathcal{U}_t, 
		\] 
		where $\Tot(\mathcal{O}_{C_t}(K_{C_t})$ is the total space of the line bundle $\mathcal{O}_{C_t}(K_{C_t})$. 
		It is actually a contraction of the zero section. 
		Thus we get a family of contractions $\tilde{\mathcal{U}} \rightarrow \mathcal{U}$. 
		Let $\eta_w \in T^1_{(U,p)}$ be the element corresponding to the deformation $\mathcal{U}$. 
		By the above description, we see that $\eta_w \in \Image f_*$, where $f:= f_0$. 
		Recall that $T^1_{(U,p)} \simeq \mathcal{O}_{U,p}^{\oplus 2}/ J_p$ for the Jacobian sub-module $J_p$ determined by the partial derivatives of 
		the defining equations of $U$. 
		Since the order of $w$ is one, we see that $\eta_w \not\in \mathfrak{m}_{U,p}^2 T^1_{(U,p)}$. 
		\end{eg}	
	
We use the pair version of the blow-down transformation as follows. 

Let $X$ be a normal variety with only rational singularities over $\mathbb{C}$, $D$ be a Cartier divisor on $X$ 
and $D = \sum_{j \in J} D_j$ be the decomposition into the irreducible components such that each $D_j$ is Cartier. 
Let $\mu \colon \tilde{X} \rightarrow X$ be a resolution of singularities of $X$. 
Let $\tilde{D} \subset \tilde{X}$ be the strict transform of $D$, 
$E$ be the exceptional locus of $\mu$ and 
$E= \sum_{i=1}^m E_i$ be the decomposition into the irreducible components. 
Since $X$ has only rational singularities, we see that $\mu_* \mathcal{O}_{\tilde{X}} \simeq \mathcal{O}_X$ and 
$R^1 \mu_* \mathcal{O}_{\tilde{X}} =0$. 

\begin{prop}\label{pairblowdownprop}(cf.\ \cite[Proposition 3.2]{kawamatacrelle}) 
Let $X, D, \tilde{X}, \tilde{D}, E$ be as above. 
Then we can define a morphism of functors 
\[
\mu_* \colon \Def_{(\tilde{X}, \tilde{D}+E)} \rightarrow \Def_{(X,D)}
\]
\end{prop}

\begin{proof}
Consider a deformation $(\tilde{\mathcal{X}}, \sum_{j \in J} \tilde{\mathcal{D}}_j + \sum_{i=1}^m \mathcal{E}_i)$ of $(X, \tilde{D} +E)$ over $A \in \Art_{\mathbb{C}}$.  
We can blow down a deformation $\tilde{\mathcal{X}}$ of $\tilde{X}$ over $A$ as in Corollary \ref{blowdowncor} 
since $R^1 \mu_* \mathcal{O}_{\tilde{X}} =0$. 

Let $\mathbf{I}_{\tilde{D}_j}, \mathbf{I}_{E_i} \subset \mathcal{O}_{\tilde{\mathcal{X}}}$ be the ideal sheaves 
of given deformations of $\tilde{D}_j, E_i$ respectively. 
We can write 
\[
\mu^* D_j = \tilde{D}_j + \sum_{i=1}^m a_{i,j} E_i 
\]
by some non-negative integers $a_{i,j}$. 
We can define a deformation of $D_j \subset X$ by the ideal 
\[
\mu_* \left( \mathbf{I}_{\tilde{D}_j} \cdot \prod_{i=1}^m \mathbf{I}_{E_i}^{a_{i,j}} \right) \subset \mathcal{O}_{\mathcal{X}}. 
\]
We can check that this ideal is $A$-flat by Proposition \ref{bldownlem} (iii) and 
\[
R^1 \mu_* \mathcal{O}_{\tilde{X}} (-\tilde{D}_j - \sum_{i=1}^m a_{i,j} E_i) = R^1 \mu_* \mu^* \mathcal{O}_X (-D_j) = 0. 
\]
\end{proof}

\begin{eg}\label{eg:blowdown}
	Let $D \subset U$ be a reduced divisor in a smooth $3$-fold $U$. 
	Let $\mu \colon \tilde{U} \rightarrow U$ be a proper birational morphism from another smooth variety $\tilde{U}$. 
	Let $\tilde{D} \subset \tilde{U}$ be the strict transform of $D$ and $E$ the $\mu$-exceptional divisor. 
	Then we can define a natural transformation $\mu_* \colon \Def_{(\tilde{U}, \tilde{D}+E)} \rightarrow \Def_{(U,D)}$ and 
	this induces a homomorphism $\mu_* \colon T^1_{(\tilde{U}, \tilde{D}+E)} \rightarrow T^1_{(U,D)}$ on the tangent spaces.   
	We use this homomorphism in the proof of Lemma \ref{lem:logkerm^2}. 
	The point is that we can define the blow-down transformation even if some irreducible component of $D$ has non-isolated singularities. 
	When $D$ has only isolated singularities, the definition of the blow-down transformation is easier (See (\ref{mu1*defn}), for example). 
	\end{eg}

\section{Deformations of elephants with isolated singularities}
In this section, we treat  deformations of a pair of a $\mathbb{Q}$-Fano $3$-fold and a member of $|{-}K_X|$ with only isolated singularities. 

\subsection{First blow-up}\label{firstblowupsection}
Consider a $\mathbb{Q}$-Fano $3$-fold $X$ and its elephant $D$ with only isolated singularities. 
By Theorem \ref{qsmqfanothmintro}, we can  assume 
$X$ with only quotient singularities and $A_{1,2}/4$-singularities for the proof of Theorem \ref{simulqsmisolthmintro}. 

Takagi proved the following theorem on singularities on general elephants of a $\mathbb{Q}$-Fano $3$-fold by using 
the standard argument. 

\begin{thm}\label{thm:takagielephant}(\cite[Proposition 1.1]{takagi2})
	Let $X$ be a $\mathbb{Q}$-Fano $3$-fold. Assume that there exists $D_0 \in |{-}K_X|$ such that 
	$D_0$ is normal near the non-Gorenstein points of $X$. 
	
	Then there exists a normal member $D \in |{-}K_X|$ such that $D$ is Du Val outside the non-Gorenstein points.   
	\end{thm}

Take a non-Du Val singularity $p$ on $D$ and its Stein neighborhood $U \subset X$.  
We first prepare lemmas on suitable weighted blow-ups of the Stein neighborhood $U$ with ``negative discrepancies''. 
By Theorem \ref{thm:takagielephant}, it is enough to consider $U$ which is analytic locally isomorphic to either of the following; 

\begin{itemize}
	\item $\mathbb{C}^3/ \mathbb{Z}_r(1,a,r-a)$ for some coprime integers $r$ and $a$, 
	\item $(x^2+y^2 +z^3+u^2=0)/ \mathbb{Z}_4 \subset \mathbb{C}^4/\mathbb{Z}_4(1,3,2,1)$.
	\end{itemize}

We argue on these explicit spaces. 
We use the same symbol $0$ for the origin of these spaces. 

\vspace{5mm}

For $U= \mathbb{C}^3/ \mathbb{Z}_r(1,a, r-a)$, we can take $1/r(1,a,r-a)$-weighted blow-up for the first blow-up as follows. 

\begin{lem}\label{nonGorwtlem}
	Let $U=\mathbb{C}^3/ \mathbb{Z}_r(1,a,r-a)$ be the quotient variety for some coprime integers $r$ and $a$ such that 
	$0< a < r$ and 
	$D \in |{-}K_U|$  an anticanonical divisor with only isolated singularity at $0 \in U$. 
	Let $\pi_U \colon V=\mathbb{C}^3 \rightarrow U$ be the quotient morphism and $\Delta := \pi_U^{-1}(D)$. 
	Assume that $\mult_0 \Delta  \ge 2$. 
	Let $\mu_1 \colon U_1 \rightarrow U$ be the weighted blow-up with weights $1/r(1,a,r-a)$ and 
	$E_1$ its exceptional divisor. 

Then we have an inequality on the discrepancy   
\[
a(E_1, U, D) \le -1. 
\]
	\end{lem}

\begin{proof}
	Let $f= \sum f_{i,j,k} x^i y^j z^k$ be the defining equation of $\Delta \subset \mathbb{C}^3$ at $0 \in \mathbb{C}^3$. 
	By the formulas in Section \ref{weightedblupsection}, we have 
	\[
	K_{U_1} = \mu_1^* K_U + \frac{1}{r}(1+a+(r-a)-r)E_1 = \mu_1^* K_U + \frac{1}{r} E_1, 
	\]
	\[
	\mu_1^* D = D_1 + \frac{m_D}{r} E_1,   
	\]
	where $D_1 \subset U_1$ is the strict transform of $D$ and  $m_D:= \min \{ i+aj+(r-a)k \mid f_{i,j,k} \neq 0 \}$. 
	We see that $m_D \ge 2$ since $\Delta$ is singular. 
	Thus we can write
	\[
	K_{U_1} +D_1 = \mu_1^* (K_U +D) + \frac{1}{r}(1-m_D) E_1  
	\]
	and $\frac{1}{r}(1-m_D) <0$. 
	Since $K_U +D$ is a Cartier divisor, we see that $\frac{1}{r}(1-m_D)$ is a negative integer. 
	Thus $\mu_1$ satisfies the required condition.  
	\end{proof}

	\vspace{5mm}

	Next we consider a neighborhood  of an $A_{1,2}/4$-singularity. 
	We describe a necessary weighted blow-up in the following example. 
	
	\begin{eg}\label{exampleA12weighted}
		Let $U:= (x^2+y^2+z^3+u^2=0)/ \mathbb{Z}_4(1,3,2,1) \subset \mathbb{C}^4/ \mathbb{Z}_4(1,3,2,1)$ 
		be a neighborhood of an $A_{1,2}/4$-singularity.
		Let $\mu_1 \colon U_1 \rightarrow U$ be the weighted blow-up with weights $1/4(1,3,2,1)$ and 
		$E_1 \subset U_1$ its exceptional divisor. 
		Let $\nu_1 \colon W_1 \rightarrow W:= \mathbb{C}^4/ \mathbb{Z}_4(1,3,2,1)$ be the weighted blow-up of weights $1/4(1,3,2,1)$ 
		and $F_1 \subset W_1$ the exceptional divisor. 
	
		$U_1$ is covered by four affine pieces 
		\begin{itemize} 
			\item $D_+(x)$: $(1+ x(y^2+z^3)+u^2=0) \subset \mathbb{C}^4$,  
			\item $D_+(y)$: $(x^2+y(1+z^3)+u^2 =0) \subset \mathbb{C}^4/ \mathbb{Z}_3(2,1,1,2)$,
			\item $D_+(z)$: $(x^2+ z(1+y^2)+ u^2=0) \subset \mathbb{C}^4/ \mathbb{Z}_2(1,1,0,1)$, 
			\item $D_+ (u)$: $(x^2+u(y^2+z^3)+1=0) \subset \mathbb{C}^4$. 
			\end{itemize}
		We can compute that $U_1$ has two ordinary double points $p_1, p_2$, a $1/2(1,1,1)$-singularity $q_2$ and a $1/3(1,2,1)$-singularity $q_3$, 
		where $p_1, q_3 \in D_+(y)$ and $p_2, q_2 \in D_+(z)$. 
		We see that 
		\[
		E_1 \simeq (x^2+u^2=0) \subset \mathbb{P}(1,3,2,1) \simeq F_1, 
		\]
		where we regard $x,y,z,u$ as coordinates on $\mathbb{P}(1,3,2,1)$.  
		We also see that $E_1$ consists of two irreducible components 
		$E_{1,1}, E_{1,2}$ corresponding to functions $x+\sqrt{-1}u$ and $x-\sqrt{-1}u$. 
		We see that $E_{1,1}$ and $E_{1,2}$ are both isomorphic to $\mathbb{P}(1,2,3)$ and 
		they intersect transversely outside $\Sing U_1= \{p_1, p_2, q_2, q_3 \}$.
		\end{eg}

	The weighted blow-up $\mu_1$ in the example satisfies the following property on the discrepancy.  
	
	\begin{lem}\label{A12-wtdblowup}
		Let $U:= (x^2+y^2+z^3+u^2=0)/ \mathbb{Z}_4(1,3,2,1)$, $\mu_1 \colon U_1 \rightarrow U$ and $E_1 \subset U_1$ 
		be as in Example \ref{exampleA12weighted}. 
		Let $D\in |{-}K_U|$ be a divisor with only isolated singularity at $0 \in D$ which is not Du Val. And let $D_1 \subset U_1$ be the strict transform of $D$. 
		
		Then we have an inequality for the discrepancy
		\[
		a(E_{1,j}, U,D) \le -1
		\] 
for $j=1,2$, that is, $K_{U_1} + D_1 + E_1 - \mu_1^*(K_U+D) $ is anti-effective.  		
		
		\end{lem}
		
		\begin{proof}
			Let $V:= (x^2+y^2+z^3+u^2=0) \subset \mathbb{C}^4$ and $\pi_U \colon V \rightarrow U$ the index one cover. 
			Let $D_V:= \pi_U^{-1}(D) \subset V$ and $h \in \mathcal{O}_{V,0}$ the local equation of 
			$D_V$. 
			Let $\bar{h} \in \mathcal{O}_{\mathbb{C}^4,0}$ be the lift of $h$ by the surjection 
			$\mathcal{O}_{\mathbb{C}^4,0} \twoheadrightarrow \mathcal{O}_{V,0}$. 
			We can assume that $\bar{h}$ is $\mathbb{Z}_4$-equivariant. 
			
			\vspace{2mm}

			\noindent{\bf (Case 1)} $\bar{h} \in  \mathfrak{m}^2_{\mathbb{C}^4,0}$. 
			
	 Let $\Delta:=(\bar{h}=0)/\mathbb{Z}_4 \subset W$ be the divisor on $W$ defined by $\bar{h}$. 
	We can write $\bar{h}= \sum h_{ijkl} x^i y^j z^k u^l$ for some $h_{ijkl} \in \mathbb{C}$. 
	We have 
	\begin{equation}\label{Delta1rel}
	\nu_1^*\Delta = \Delta_1 + m_1 F_1, 
	\end{equation}
	where $\Delta_1 \subset W_1$ is the strict transform of $\Delta$ and 
	\begin{equation*}
	m_1 = \min \{ (i+3j+2k+l)/4 \mid h_{ijkl} \neq 0 \}.
	\end{equation*}
	Since $\mult_0 (\bar{h}) \ge 2$, we have $m_1 \ge 1/2$. 
	Thus, by restricting (\ref{Delta1rel}) to $U_1$, we obtain 
	\begin{equation}\label{D1rel}
	\mu_1^* D = D_1 + l_1 E_1  
	\end{equation}
      for some rational number $l_1 \ge 1/2$. 
	 
	 We can compute  
	 \[
	 K_{U_1} = \mu_1^* K_U + \frac{1}{4} E_1.  
	 \]
	Thus we obtain 
	\[
	K_{U_1} +D_1 = \mu_1^*(K_U +D) + (\frac{1}{4} - l_1) E_1. 
	\]
	Since $K_U +D$ is a Cartier divisor, we see that $1/4 - l_1$ is a negative integer. 
	Thus we get the required inequality of the discrepancy.
				
				\vspace{2mm}
				
				\noindent{\bf (Case 2)} 
				$\bar{h} \in \mathfrak{m}_{\mathbb{C}^4,0}\setminus \mathfrak{m}_{\mathbb{C}^4,0}^2$. 
				
				Let $g \in \mathbb{Z}_4$ be the generator.  
				Since $h$ is a $\mathbb{Z}_4$-eigenfunction such that $g \cdot h = \sqrt{-1} h$, we can write 
				$\bar{h}= ax+bu+h_1$, where $a,b \in \mathbb{C}$ and 
				\[
				h_1 := \sum h_{ijkl} x^i y^j z^k u^l \in \mathcal{O}_{\mathbb{C}^4,0}
				\]
				satisfies that 
				\begin{equation}\label{eqn:weightineq}
				i+3j+2k+l \ge 5
				\end{equation}
				 when $h_{ijkl} \neq 0$. 
				
				Since $D$ has a non-Du Val singularity at $0 \in D$, 
				we can write $\bar{h}= x + \zeta_4 u +h_1$, where $\zeta_4 = \pm \sqrt{-1}$. 
				Otherwise we see that $D_V$ has a Du Val singularity of type $A_2$ 
				at $0$ and $D$ also has a Du Val singularity at $0$. 
				
				Hence we have 
				$D= \Delta \cap H \subset \mathbb{C}^4/ \mathbb{Z}_4(1,3,2,1)$, 
				where 
				\[
				\Delta:= (y^2 +z^3 + 2 \zeta_4 h_1 u + h_1^2=0)/ \mathbb{Z}_4,
				\]
				\[
				H:= (x + \zeta_4 u + h_1=0) / \mathbb{Z}_4. 
				\]
				Let $\Delta_1 \subset W_1$ be the strict transform of $\Delta$. 
				Then we have 
				\begin{equation}\label{eqn:Deltapullback}
				\nu_1^* \Delta = \Delta_1 + m_1 F_1,
				\end{equation}
				where 
				\[
				m_1 := \min \left\{ \frac{3}{2}, \left\{ \frac{1 + i +3j +2k +l}{4} \mid h_{ijkl} \neq 0 \right\} \right\}. 	
				\]
				We see that $m_1 \ge \frac{3}{2}$ by (\ref{eqn:weightineq}). 
				
				Let $H_1 \subset W_1$ be the strict transform of $H$. 
			Let $\nu_H:= \nu_1|_{H_1} \colon H_1 \rightarrow H$ be the induced birational morphism 
			and $E_H:= \nu_H^{-1}(0)$ be the exceptional divisor.
				By restricting (\ref{eqn:Deltapullback}) to $H_1$, we obtain 
				\[
				\nu_H^* D = D_1 + m_2 E_H 
				\]
			for some $m_2 \ge m_1$. 
			Since we have $K_{H_1} = \nu_H^* K_H + \frac{1}{2}E_H$, we obtain 
			\[
			K_{H_1} +D_1 = \nu_H^* (K_H +D) + (\frac{1}{2}- m_2) E_H.  
			\]
			By restricting this to $D_1$, we obtain 
			\begin{equation}\label{eq:ramifD1}
			K_{D_1} = \nu_D^* K_D + (\frac{1}{2} - m_2)E_D, 
			\end{equation}
			where $E_D:= E_H \cap D_1$ is the exceptional divisor of $\mu_D = \nu_H|_{D_1} \colon D_1 \rightarrow D$.  
			Note that we also have  $\mu_D = \mu_1|_{D_1}$. 
			Since $1/2 - m_2 \le -1$, we obtain the claim as follows; 
			since we have 
			\begin{equation}\label{eq:ramifU1D1}
			K_{U_1} +D_1 = \mu_1^*(K_U +D) + \sum_{j=1}^2 a(E_{1,j}, U, D) E_{1,j} 
			\end{equation}
			and $K_{U_1} +D_1$ is a $\mathbb{Q}$-Cartier divisor, 
			we see that $a(E_{1,1}, U, D) = a(E_{1,2}, U, D)$. 
			Otherwise $\sum_{j=1}^2 a(E_{1,j}, U, D) E_{1,j}$ is not $\mathbb{Q}$-Cartier at the ordinary double points.   
			Moreover we see that $a(E_{1,j}, U, D) \in \mathbb{Z}$ since 
			$K_U +D$ is a Cartier divisor. 
			By restricting (\ref{eq:ramifU1D1}) to $D_1$ and comparing that with (\ref{eq:ramifD1}), 
			we see that $a(E_{1,j}, U, D) \le -1$.

		\end{proof}

\subsection{Second blow-up}\label{secondblowupsect}

Let $U_1 \rightarrow U$ be either one of the weighted blow-ups constructed in Section \ref{firstblowupsection}. 
We use the same notation as Section \ref{firstblowupsection}.

We construct a useful resolution $U_2 \rightarrow U_1$ of $(U_1, D_1 +E_1)$ as follows.

\begin{lem}\label{U_2lem}
	Let $\mu_1 \colon U_1 \rightarrow U, D_1$ and $E_1$ be those as in Section \ref{firstblowupsection}. 
	
	Then there exists a projective birational morphism $\mu_{12} \colon U_2 \rightarrow U_1$ and 
	a $0$-dimensional subset $Z \subset U_1$ with the following properties;
	\begin{enumerate}
		\item[(i)] $U_2$ is smooth and $\mu_{12}^{-1}(D_1 \cup E_1)$ has SNC support. 
		\item[(ii)] $\mu_{12}$ is an isomorphism over $U_1 \setminus ( (D_1 \cap E_1) \cup \Sing U_1)$. 
		\item[(iii)] $\mu_{12}' \colon U_2' := \mu_{12}^{-1}(U_1 \setminus Z) \rightarrow U_1':= U_1 \setminus Z$ can be written as 
		a composition 
		\[
		\mu_{12}':  U_2' = V'_{k} \stackrel{f'_{k-1}}{\rightarrow} V'_{k-1} \rightarrow \cdots \rightarrow V'_2 \stackrel{f'_1}{\rightarrow} 
		V'_1 = U'_1, 
		\]
		where $f'_i \colon V'_{i+1} \rightarrow V'_i$ is an isomorphism or a blow-up of a smooth curve $Z'_i$ with either of the followings;  
		\begin{itemize}
		\item If the strict transform $\Delta'_i \subset V'_i$ of $\Delta_1':= D_1 \cap V_1' \subset V_1'$ is singular, we have $Z'_i \subset \Sing \Delta'_i$.  
		\item If $\Delta'_i$ is smooth, we have $Z'_i \subset \Delta'_i \cap F'_i$, where 
		 $F'_i:= (f'_{i,1})^{-1}(E_1')$ is the inverse image of $E_1':= E_1 \cap U_1'$ by   
		$f'_{i,1}:= f'_1 \circ \cdots \circ f'_{i-1} \colon V'_i \rightarrow V'_1$. 
		\end{itemize}
		\end{enumerate}
	
	As a consequence,  
	the discrepancies satisfy  
	\[
	a(E_{2,j}', U_1', D_1') \le 0 
	\] 
	for $D_1':= D_1 \cap U_1'$ and all $\mu_{12}'$-exceptional divisors $E_{2,j}' \subset U_2'$. 
	\end{lem}
	
	\begin{proof}
		By the construction of $\mu_1 \colon U_1 \rightarrow U$, we see that $U_1$ has only isolated cyclic quotient singularities and 
		ordinary double points. We also see that ${\rm Nsnc} (E_1) \subset \Sing U_1$, 
		where ${\rm Nsnc}(E_1)$ is the non-SNC locus of $E_1$. 
		Let $\nu_1 \colon V_1 \rightarrow U_1$ be a composition of blow-ups of smooth centers such that 
		$V_1$ is smooth, $F_1:= \nu_1^{-1}(E_1)$ is a SNC divisor, 
		and it induces an isomorphism $\nu_1^{-1}(U_1 \setminus \Sing U_1) \stackrel{\sim}{\rightarrow} U_1 \setminus \Sing U_1$. 
		Let $\Delta_1 \subset V_1$ be the strict transform of $D_1$. 
		Then we see that the non-SNC locus of $\Delta_1 \cup F_1$ is contained in $\Delta_1 \cap F_1$. 
		We can construct a composition of smooth center blow-ups as 
		\[
		f_{k,1} \colon V_k \stackrel{f_{k-1}}{\rightarrow} V_{k-1} \rightarrow \cdots \rightarrow V_2 \stackrel{f_1}{\rightarrow} V_1,   
		\] 
   	 where $f_i \colon V_{i+1} \rightarrow V_i$ is a blow-up of a smooth center $Z_i \subset V_i$ such that, 
   	 for each $i$, 
   	 \begin{itemize}
   		 \item $\Delta_i \subset V_i$ is the strict transform of $\Delta_1$, 
   		 \item $f_{i,1} := f_{1} \circ \cdots \circ f_{i-1} \colon V_i \rightarrow V_1$ is a composition and 
		 $F_i:= f_{i,1}^{-1}(F_1) \subset V_i$, 
   		 \end{itemize} 
   		 then, for each $i$, 
   		 \begin{enumerate}
   			 \item[(i)'] $Z_i$ and $F_i$ intersect transversely, 
   			 \item[(ii)'] $Z_i \subset \Sing \Delta_i$ or $\Delta_i$ is smooth and $Z_i \subset \Delta_i \cap F_i$
   			 \item[(iii)']  $\Delta_k \cup F_k$ is a SNC divisor. 
   			 \end{enumerate}
		We can construct this resolution by \cite[Theorem A.1]{BierstoneMilmanMinimal1}, for example.
		
		Let $U_2:= V_k$, $\mu_{12}:= \nu_1 \circ f_{k,1}$ and 
		\[
		Z:= \nu_1 \left( \bigcup_{\dim f_{i,1}(Z_i) =0} f_{i,1}(Z_i) \right) \cup \Sing U_1. 
		\]
		We see that these satisfy the property (i) in the statement by the properties (i)' and (iii)'  of $f_{k,1}$. 
		We see the property (ii) since  the morphism $\nu_1$ is an isomorphism 
		outside $\Sing U_1$ and $Z_i$ is contained in the inverse image of $D_1 \cap E_1$ by the condition (ii)'. 
		We see the property (iii) by the property (ii)' of $f_{k,1}$. 
	
	We can check the inequality $a(E_{2,j}', U_1', D_1') \le 0$ as follows; 
	   For $i \ge j$, let $f'_{i,j}:= f'_{j} \circ \cdots \circ f'_{i-1} : V'_i \rightarrow V'_j $ be the composition. 
	    We have an equality 
	   \begin{multline}
	   \sum_j a(E_{2,j}', U_1', D_1') E_{2,j}' = K_{V_k'} +\Delta'_k - (\mu'_{12})^* (K_{V_1'} +\Delta'_1)  \\ 
	   = \sum_{i=1}^{k-1} (f'_{k, i+1})^* (K_{V_{i+1}'} +\Delta'_{i+1} - (f'_{i})^* (K_{V'_i} + \Delta'_i)).  
	   \end{multline}
	  We also have $K_{V'_{i+1}} +\Delta'_{i+1} - (f'_i)^* (K_{V'_i} + \Delta'_i) = 
	  (1- \mult_{Z'_{i}} (\Delta'_i)) (f'_{i})^{-1}(Z'_{i})$ and 
	 $1- \mult_{Z'_{i}} (\Delta'_i) \le 0$.  
	Thus  we see that $a(E_{2,j}', U_1', D_1') \le 0$ for each $j$. 
	   
	  We finish the proof of Lemma \ref{U_2lem}.  
		\end{proof}

\subsection{The image of the blow-down morphism}		
		
		Let $U$ and $D \in |{-}K_U|$ with an isolated singularity at $0 \in D$ be as in Section \ref{firstblowupsection}. 
		We study the image of the blow-down morphism $(\mu_1)_* \colon T^1_{(U_1, D_1 +E_1)} \rightarrow T^1_{(U,D)}$ as in Example \ref{eg:blowdown}. 
		 
	 Let $\pi \colon V \rightarrow U$ be the index one cover and $\Delta:= \pi^{-1}(D)$. 
	 We can assume that $\Delta = (f=0) \subset V$.   
		By Proposition \ref{T^1invprop}, we have 
		\[
		T^1_{(U,D)} \simeq (T^1_{(V,\Delta)})^{\mathbb{Z}_r} 
		\]
		and  regard $T^1_{(U,D)} \subset T^1_{(V, \Delta)}$. 
		Let 
		\begin{equation}\label{m2definition}
		\mathfrak{m}^2 T^1_{(U,D)} := \mathfrak{m}_{V,0}^2 T^1_{(V,\Delta)} \cap T^1_{(U,D)}
		\end{equation} 
		be the set of deformations induced by functions with multiplicity 2 or more.
		
		\vspace{5mm}
		
		\noindent(I) First consider the case where $0 \in U$ is a quotient singularity.  
		Since $V$ is smooth, we also have $T^1_{(V,\Delta)} \simeq T^1_{\Delta} \simeq \mathcal{O}_{V,0}/J_{f,0}$ 
		for the Jacobian ideal $J_{f,0} \subset \mathcal{O}_{V,0}$. 
		Thus $T^1_{(V,\Delta)}$ has a $\mathcal{O}_{V,0}$-module structure and we fix an $\mathcal{O}_{V,0}$-module 
		homomorphism 
		\[
		\varepsilon \colon \mathcal{O}_{V,0} \rightarrow T^1_{(V,\Delta)} 
		\]
		such that, for $g \in \mathcal{O}_{V,0}$, an element $\varepsilon(g) \in T^1_{(V,\Delta)}$ 
		is a deformation $(f +tg=0) \subset V \times \Spec \mathbb{C}[t]/(t^2)$ of $V$.

Since $D$ has an isolated singularity at $0 \in U$, we obtain $T^1_{(U,D)} \simeq T^1_{(U', D')}$ by Lemma \ref{codim3pair}, 
where $U':= U \setminus 0$ and $D' := D \cap U$. 

Let $\mu_1 \colon U_1 \rightarrow U$ be one of the weighted blow-ups constructed in Section \ref{firstblowupsection}. 
We can define the {\it blow-down morphism} $(\mu_1)_* \colon T^1_{(U_1, D_1+E_1)} \rightarrow T^1_{(U,D)}$ as a composition 
\begin{equation}\label{mu1*defn}
(\mu_1)_* \colon T^1_{(U_1, D_1 +E_1)} \stackrel{\iota_1^*}{\rightarrow} T^1_{(U',D')} \stackrel{\simeq}{\rightarrow} T^1_{(U,D)},  
\end{equation}
where $\iota_1^*$ is the restriction by an open immersion $\iota_1 \colon U' \simeq U_1 \setminus E_1 \hookrightarrow U_1$. 
This homomorphism $(\mu_1)_*$ is same as the homomorphism introduced in Proposition \ref{pairblowdownprop}.

We have a relation on $\Image (\mu_1)_* \subset T^1_{(U,D)}$ as follows. 
		
\begin{lem}\label{blowdownlem}
	Let $U:= \mathbb{C}^3/ \mathbb{Z}_r(1,a,r-a)$ for some coprime positive integers $r$ and $a$.
Let $\mu_1 \colon U_1 \rightarrow U, D_1$ and $E_1$ be as in Section \ref{firstblowupsection}.  

Then we have the following.  
\begin{enumerate}
	\item[(i)] $T^1_{U_1} =0$. 
\item[(ii)] $\Image (\mu_1)_* \subset \mathfrak{m}^2  T^1_{(U,D)}$. 
\end{enumerate}
\end{lem}

\begin{proof}
	(i) Since $U_1$ has only isolated quotient singularities, we have an isomorphism 
	\[
	T^1_{U_1} \simeq H^1(U_1, \Theta_{U_1}) \simeq H^1(U_1, (\iota_1)_*(\Omega^2_{U'_1}(-K_{U'_1})), 
	\]
	where $\iota_1 \colon U'_1 \hookrightarrow U_1$ is the smooth part. 
	Since $-K_{U_1}$ is $\mu_1$-ample in each case, we see that $H^1(U_1, (\iota_1)_*(\Omega^2_{U'_1}(-K_{U'_1}))=0$ by Lemma \ref{QFujinotoric}. 
	Thus we finish the proof of (i).

	(ii) Take $\eta_1 \in T^1_{(U_1,D_1 +E_1)}$. 
	We have an exact sequence 
	\begin{equation}\label{divisorcomesequence}
	H^0(U_1, \mathcal{O}_{U_1}(D_1)) \rightarrow H^0(D_1, \mathcal{N}_{D_1/U_1}) \rightarrow H^1(U_1, \mathcal{O}_{U_1})=0. 
	\end{equation}
	Hence the deformation of $D_1$ induced by $\eta_1$ comes from some member of the linear system $|D_1|$. 
	In particular, it can be extended to a deformation over a unit disc $\Delta^1$. 
	We also obtain $H^0(E_1, \mathcal{N}_{E_1/ U_1})=0$ since $-E_1$ is $\mu_1$-ample. 
	Hence $\eta_1$ induces a trivial deformation of $E_1$ over a unit disc. 
	
	By these arguments and (i), the first order deformation $\eta_1$ can be extended to 
	a deformation $(\mathcal{U}_1, \mathcal{D}_1 + \mathcal{E}_1) \rightarrow \Delta^1$ of $(U_1, D_1 +E_1)$
	over a unit disc $\Delta^1$ such that $\mathcal{U}_1 \simeq U_1 \times \Delta^1$. 
	   By taking its image by $\mu_1 \times \id \colon U_1 \times \Delta^1 \rightarrow U \times \Delta^1$, 
	   we obtain a deformation $(\mathcal{U}, \mathcal{D}) \rightarrow \Delta^1$ of $(U,D)$. 
	  	
	   Let $m_1$ be a rational number such that 
	   \[
	   (\mu_1 \times \id)^* (r \mathcal{D}) = r \mathcal{D}_1 + rm_1 \mathcal{E}_1.
	   \]
	   For $t \in \Delta^1$, let $\mathcal{D}_t, \mathcal{D}_{1,t}$ be the fibers of $\mathcal{D}, \mathcal{D}_1$ over $t$ and
	     $m_{1,t}$  a rational number such that 
	    $\mu_1^* \mathcal{D}_t = \mathcal{D}_{1,t} + m_{1,t} E_1$. 
	   The above relations imply that $m_{1,t}$ is invariant for all $t \in \Delta^1$. 
	   
	 Suppose that there exists $\eta_1 \in T^1_{(U_1, D_1 +E_1)}$ such that 
	 \begin{equation}\label{m^2hairikatei}
	 (\mu_1)_* (\eta_1) \in T^1_{(U,D)} \setminus \mathfrak{m}^2 T^1_{(U,D)}.
	 \end{equation} We use the inclusion $T^1_{(U,D)} \subset T^1_{(V,\Delta)}$ as above. 
	  Take $h_1 \in \mathcal{O}_{V,0}$ such that $\varepsilon(h_1) = (\mu_1)_* (\eta_1)$. By the condition (\ref{m^2hairikatei}), 
	  we obtain $\mult_0 h_1 \le 1$.
	 Then we see that $m_{1,t} \le (1/r) \max \{1,a, r-a \}$ for $t \neq 0$ by the formula (\ref{stricttrfwtblup}). 
	 However we see that $m_{1,0} \ge 1 + 1/r$ by the calculation in the proof of Lemma \ref{nonGorwtlem}.  
	 This is a contradiction. 
	  
	 Hence we finish the proof of (ii). 
	\end{proof}

\vspace{5mm}

\noindent(II) Next we consider a neighborhood of an $A_{1,2}/4$-singularity.  
Let $U:= (x^2+y^2+z^3+u^2=0)/ \mathbb{Z}_4(1,3,2,1)$ and 
$\mu_1 \colon U_1 \rightarrow U$ the weighted blow-up as in Lemma \ref{A12-wtdblowup}. 
	Let $W:= \mathbb{C}^4/\mathbb{Z}_4(1,3,2,1)$ and $\nu_1 \colon W_1 \rightarrow W$ the weighted blow-up 
	as in Lemma \ref{A12-wtdblowup}. 
We first want to show that deformations of $U_1$ comes from embedded deformations of $U_1 \subset W_1$.  

	Let $\mathcal{I}_{U_1} \subset \mathcal{O}_{W_1}$ be the ideal sheaf of the closed subscheme $U_1 \subset W_1$. 
	Let $U_1' \subset U_1$ and $W'_1$ be the smooth parts of $U_1$ and $W_1$ respectively. 
	Note that $U_1'= W_1' \cap U_1$. 
Let $\mathcal{I}_{U'_1} \subset \mathcal{O}_{W_1'}$ be the ideal sheaf of $U'_1 \subset W_1'$. 
	We have an exact sequence 
	\[
	0 \rightarrow \mathcal{I}_{U'_1}/ \mathcal{I}_{U'_1}^2 \rightarrow \Omega^1_{W'_1}|_{U'_1} \rightarrow 
	\Omega^1_{U'_1} \rightarrow 0.
	\]

By taking the push-forward of the above sequence by the open immersion $\iota_1 \colon U'_1 \hookrightarrow U_1$, 
we obtain an exact sequence 
\[
	0 \rightarrow  (\mathcal{I}_{U_1}/ \mathcal{I}_{U_1}^2)^{**} \rightarrow  (\Omega^1_{W_1}|_{U_1})^{**} \rightarrow 
	 (\Omega^1_{U_1})^{**} \rightarrow 0,     
\]
where ${}^{**}$ is the double dual. 
The surjectivity follows since $(\iota_1)_* \mathcal{I}_{U'_1}/ \mathcal{I}_{U'_1}^2$ is a Cohen-Macaulay sheaf and 
it implies $R^1 (\iota_1)_* \mathcal{I}_{U'_1}/ \mathcal{I}_{U'_1}^2 =0$.

This induces an exact sequence 
\begin{equation}\label{exactembedabstdeform}
H^0(U_1, \mathcal{N}_{U_1/ W_1}) \stackrel{e_1}{\rightarrow} \Ext^1((\Omega^1_{U_1})^{**}, \mathcal{O}_{U_1}) \rightarrow 
\Ext^1((\Omega^1_{W_1}|_{U_1})^{**}, \mathcal{O}_{U_1}).  
\end{equation}
By Lemmas  \ref{codim3Omega1**} and \ref{codim3pair}, we obtain  
\[
\Ext^1((\Omega^1_{U_1})^{**}, \mathcal{O}_{U_1}) \simeq \Ext^1(\Omega^1_{U'_1}, \mathcal{O}_{U'_1}) \simeq T^1_{U_1}. 
\]
Thus the homomorphism $e_1$ in (\ref{exactembedabstdeform}) sends an embedded deformation of $U_1 \subset W_1$ to 
the corresponding deformation of $U_1$. 

We have the following proposition on the surjectivity of $e_1$. 

\begin{lem}\label{surjembabs}
	Let $U, U_1, W, W_1$ be as above. 
	Then we have 
	\[
	\Ext^1((\Omega^1_{W_1}|_{U_1})^{**}, \mathcal{O}_{U_1}) =0. 
	\]
	In particular, we see that $e_1$ is surjective by the sequence (\ref{exactembedabstdeform}). 
	\end{lem}

\begin{proof}
	The local-to-global spectral sequence induces an exact sequence 
	\[
	0 \rightarrow H^1(U_1, (\Omega^1_{W_1}|_{U_1})^*) 
	\rightarrow \Ext^1((\Omega^1_{W_1}|_{U_1})^{**},\mathcal{O}_{U_1} ) 
	\rightarrow H^0(U_1, \underline{\Ext}^1((\Omega^1_{W_1}|_{U_1})^{**},\mathcal{O}_{U_1}) ),  
	\]
	where $\underline{\Ext}^1$ is the sheaf of Ext groups. 
	Thus it is enough to show the second and fourth terms are zero. 
	
	First we show that $H^1(U_1, (\Omega^1_{W_1}|_{U_1})^*) =0$. 
	Let $\iota \colon W'_1 \hookrightarrow W_1$ be the open immersion. 
	We can see that the sheaf $\iota_* \Theta_{W_1'}(-U'_1)$ is Cohen-Macaulay as follows; 
	On $D_+(y) \simeq \mathbb{C}^4/\mathbb{Z}_3(2,1,1,2)$, let $\pi_y \colon \mathbb{C}^4 \rightarrow D_+(y)$ 
	be the quotient morphism. 
	We see that  $\iota_* \Theta_{W_1'}(-U'_1)|_{D_+(y)}$ is Cohen-Macaulay since it is the $\mathbb{Z}_3$-invariant part of the sheaf 
	$\Theta_{\mathbb{C}^4}\otimes \mathcal{O}_{\mathbb{C}^4}(- \pi_y^{-1}(U_1 \cap D_+(y)))$. 
	Similarly, on $D_+(z) \simeq \mathbb{C}^4/ \mathbb{Z}_2(1,1,0,1)$, 
	we see that $\iota_* \Theta_{W_1'}(-U'_1)|_{D_+(y)}$ is Cohen-Macaulay. 
	Since $D_+(x)$ and $D_+(u)$ are smooth, we see that $\iota_* \Theta_{W_1'}(-U'_1)$ is Cohen-Macaulay on $W_1$.  
	 
	By this Cohen-Macaulayness and Proposition \ref{openpushforwardexact}, we obtain an exact sequence 
	\[
	0 \rightarrow \iota_* \Theta_{W_1'}(-U'_1) \rightarrow  \Theta_{W_1} \rightarrow \iota_* (\Theta_{W'_1}|_{U'_1}) \rightarrow 0. 
	\]
	Thus, by $\iota_* (\Theta_{W'_1}|_{U'_1}) \simeq (\Omega^1_{W_1}|_{U_1})^*$, 
	 we obtain an exact sequence 
	\begin{equation}\label{cohomexact}
	H^1(W_1, \Theta_{W_1}) \rightarrow H^1(U_1, (\Omega^1_{W_1}|_{U_1})^*) \rightarrow H^2(W_1, \iota_* \Theta_{W_1'}(-U'_1)).  
	\end{equation}
	We see that 
	\[
	H^1(W_1, \Theta_{W_1}) \simeq H^1(W_1, \iota_* (\Omega^3_{W'_1}(-K_{W'_1})) =0
	\]
	by Lemma \ref{QFujinotoric} since $-K_{W_1} \equiv_{\nu_1} -3/4 F_1$ is $\nu_1$-ample. Here $\equiv_{\nu_1}$ means the numerical equivalence over $W$. 
	Similarly, we see that 
	\[
	H^2(W_1, \iota_* \Theta_{W_1'}(-U'_1)) \simeq H^2(W_1, \iota_* (\Omega^3_{W'_1}(-K_{W'_1}-U'_1))) =0 
	\]
	since $-K_{W_1}-U_1 \equiv_{\nu_1} -1/4 F_1$ is $\nu_1$-ample. 
	Thus we obtain $H^1(U_1, (\Omega^1_{W_1}|_{U_1})^*) =0$ by the exact sequence (\ref{cohomexact}).
	
	Next we shall show that $H^0(U_1, \underline{\Ext}^1((\Omega^1_{W_1}|_{U_1})^{**},\mathcal{O}_{U_1}) ) =0$. 
We can compute that $\Sing W_1 \cap U_1 = \{q_2, q_3 \}$ consists of two quotient singularities as described in Example \ref{exampleA12weighted}. 
Hence it is enough to check $\underline{\Ext}^1((\Omega^1_{W_1}|_{U_1})^{**},\mathcal{O}_{U_1})_{q_i} =0$ for $i=2,3$.   
We have an exact sequence 
\[
\underline{\Ext}^1((\Omega^1_{U_1})^{**} ,\mathcal{O}_{U_1})_{q_i}  
 \rightarrow 
\underline{\Ext}^1((\Omega^1_{W_1}|_{U_1})^{**},\mathcal{O}_{U_1})_{q_i} \rightarrow 
\underline{\Ext}^1((\mathcal{I}_{U_1}/ \mathcal{I}_{U_1}^2)^{**},\mathcal{O}_{U_1})_{q_i} 
\] 
for $i =2,3$.  
By Lemma \ref{codim3Omega1**}, we obtain 
\[
\underline{\Ext}^1((\mathcal{I}_{U_1}/ \mathcal{I}_{U_1}^2)^{**},\mathcal{O}_{U_1})_{q_i} \simeq 
\iota_*(\underline{\Ext}^1(\mathcal{I}_{U'_1}/ \mathcal{I}_{U'_1}^2,\mathcal{O}_{U'_1}))_{q_i}
\simeq H^2_{q_i}(U_1, \mathcal{N}_{U_1/W_1}) =0
\]
since $\depth_{q_i} \mathcal{N}_{U_1/W_1} =3$. 
Hence we obtain $H^0(U_1, \underline{\Ext}^1((\Omega^1_{W_1}|_{U_1})^{**},\mathcal{O}_{U_1}) ) =0$. 

Thus we finish the proof of Lemma \ref{surjembabs}.

	\end{proof}

For a neighborhood $U$ of a $A_{1, 2}/4$-singularity, we also have the blow-down morphism 
$
(\mu_1)_* \colon T^1_{(U_1, D_1 +E_1)} \rightarrow T^1_{(U,D)} 
$
as in (\ref{mu1*defn}). 
Let $\pr_U \colon T^1_{(U,D)} \rightarrow T^1_U$ and $\pr_{U_1} \colon T^1_{(U_1, D_1 +E_1)} \rightarrow T^1_{U_1}$  
be the forgetful homomorphisms. 
Then we have the following claim on $\Image (\mu_1)_*$. 

\begin{lem}\label{lem:A12prU}
	Under the above settings, we have $\pr_U (\Image (\mu_1)_*) =0 \subset T^1_U$. 
	\end{lem}

\begin{proof}
	Let $\eta_1 \in T^1_{(U_1, D_1 +E_1)}$ be a first order deformation of $(U_1, D_1 +E_1)$.

	We have an exact sequence 
	\[
	H^0(W_1, \mathcal{O}_{W_1}(U_1)) \rightarrow H^0(U_1, \mathcal{N}_{U_1/W_1}) \rightarrow H^1(W_1, \mathcal{O}_{W_1}).   
	\]
	Since $H^1(W_1, \mathcal{O}_{W_1})=0$, the embedded deformation of $U_1 \subset W_1$ comes from some member of the linear system $|\mathcal{O}_{W_1}(U_1)|$ 
	and can be extended over a unit disk $\Delta^1$. 
	By this and Lemma \ref{surjembabs}, we see that a deformation $\pr_{U_1}(\eta_1) \in T^1_{U_1}$
	 is induced by a deformation $\mathcal{U}_1 \subset W_1 \times \Delta^1$ 
	of the embedding $U_1 \subset W_1$. 
	Let 
	\[
	\mathcal{U}:= (\nu_1 \times \id)(\mathcal{U}_1) \subset W \times \Delta^1. 
	\] 
	Let $\mathcal{F}_1:= F_1 \times \Delta^1$, where $F_1 \subset W_1$ is the $\nu_1$-exceptional divisor. 
	We have a relation 
	\[
	(\nu_1 \times \id)^* \mathcal{U} = \mathcal{U}_1 + m_1 \mathcal{F}_1
	\]
	for some $m_1 \in \mathbb{Q}_{>0}$. 
	For $t \in \Delta^1$, this induces a relation on the fibers over $t$  
	\[
	\nu_1^* \mathcal{U}_t = \mathcal{U}_{1,t} + m_{1,t} F_1,  
	\]
	where $m_{1,t}=m_1$.
	
	If $\pr_U \left( (\mu_1)_*(\eta_1) \right) \neq 0 \in T^1_{U_1}$, we see that $\mathcal{U}$ is 
	a $\mathbb{Q}$-smoothing of $U$ since $T^1_{U} \simeq \mathbb{C}$ is generated by a 
	$\mathbb{Q}$-smoothing direction.   
	Thus we see that $m_{1,t} < m_{1,0}$ as in the proof of Lemma \ref{blowdownlem} (ii).  
	 This is a contradiction. Thus we finish the proof of Lemma \ref{lem:A12prU}.
	\end{proof}

\subsection{Kernel of the coboundary map}

Let $0 \in U$ be a neighborhood of a quotient singularity or an $A_{1,2}/4$-singularity and $D \in |{-}K_U|$ be a member with an isolated singularity. 
Let $U_1, U_2, D_2$ and so on be as in Lemma \ref{U_2lem}. 
Let $\mu_2:= \mu_1 \circ \mu_{12} \colon U_2 \rightarrow U$ and $E_2:= \mu_2^{-1}(0)$ the $\mu_2$-exceptional divisor. 

Since 
$U_2 \setminus E_2 \simeq U \setminus 0=:U'$, we have the coboundary map   
\begin{equation}\label{phiUdefn}
\phi_U \colon H^1(U',\Omega^2_{U'}(\log D')) \rightarrow H^2_{E_2}(U_2, \Omega^2_{U_2}(\log D_2 +E_2)),  
\end{equation}
where $D':= D \cap U'$. 
We fix an isomorphism $S_D \colon \mathcal{O}_U(-K_U -D) \simeq \mathcal{O}_U$ and it induces an isomorphism 
	 \begin{equation}\label{varphiSDdefn} 
	 \varphi_{S_D} \colon T^1_{(U,D)} \rightarrow H^1(U',\Theta_{U'}(- \log D')) \rightarrow H^1(U',\Omega^2_{U'}(\log D')). 
	 \end{equation}

We have the following lemma on the kernel of the above coboundary map. 

 \begin{lem}\label{lem:logkerm^2}
	Let $\phi_U$ be as in (\ref{phiUdefn}) and we use the same notations as above.
	
\begin{enumerate}
	\item[(i)] We have 
	\begin{equation}\label{eq:lemlogkerm^2}
	\Ker \phi_U \subset \varphi_{S_D}( \Image (\mu_1)_*  )
	\end{equation}
	 and thus, 
	by Lemma \ref{lem:A12prU}, we have  
	$\pr_U (\varphi_{S_D}^{-1}(\Ker \phi_U)) =0$. \\ 
	  In particular, we have $\phi_U \neq 0$.  
	\item[(ii)] Assume that $U$ has only quotient singularity. 
	 Then we have 
	 \[
	 \Ker \phi_U \subset \varphi_{S_D}(\mathfrak{m}^2 T^1_{(U, D)}).
	 \]   
	 \end{enumerate}
	 
	 \end{lem}
	
	\begin{proof}
	Let $E_{12} \subset U_2$ be the $\mu_{12}$-exceptional locus. 
	Let $\mu_{12} \colon U_2 \rightarrow U_1$ be the birational morphism as in Lemma 
\ref{U_2lem} and $Z \subset U_1$ be the finite points as in Lemma \ref{U_2lem}. 
	Let $U_1':= U_1 \setminus Z$,  	
	 $U_2':= \mu_{12}^{-1}(U_1')$ and $U''_1:= U_1 \setminus (\mu_{12}(E_{12}) \cup Z)$. 
	We have the following relation; 
\begin{equation}\label{openimmdiagram}
\xymatrix{
 & U_2' \ar@{^{(}->}[r]^{\iota_2} \ar[d]^{\mu'_{12}} & U_2 \ar[d]^{\mu_{12}} \\  
U_1'' \ar@{^{(}->}[ru]^{\iota_{12}} \ar@{^{(}->}[r] & U_1'  \ar@{^{(}->}[r] & U_1 \ar[d]^{\mu_1} \\ 
 & U' \ar@{^{(}->}[lu]^{\iota_1} \ar@{^{(}->}[r] & U.
}
\end{equation}	

	Set $D_j':= D_j \cap U_j', E_j':= E_j \cap U'_j$ for $j=1,2$.

	Let $G'_2$ be a divisor on $U_2'$ supported on $E_2'$ such that 
	\[
	\{ -(K_{U_2}+ D_2 +E_2) + \mu_{2}^*(K_{U}+D)\}|_{U'_2} \sim G'_2. 
	\]
	We see that $G_2'$ is effective since we have 
	\begin{multline}
	G_2' = -E_2' + \left\{ -(K_{U_2'} +D_2') + (\mu'_{12})^* (K_{U'_1}+D'_1) \right\} \\
	+ (\mu'_{12})^* \left\{ -(K_{U_1} +D_1) + \mu_1^* (K_U +D) \right\}|_{U'_1} 
	 \ge -E'_2 +0 + (\mu'_{12})^* E'_1 \ge 0
	\end{multline}
	by Lemmas \ref{nonGorwtlem}, \ref{A12-wtdblowup} in Section \ref{firstblowupsection} and Lemma \ref{U_2lem}. 
	Set $G''_1 := G'_2 \cap U''_1$. 
	Since we have an open immersion  
	\[
	\iota \colon U'= U \setminus 0 \simeq U_2 \setminus \mu_2^{-1}(0) \hookrightarrow U_2,  
	\]
		 we obtain the following commutative diagram; 
		\begin{equation}
			\xymatrix{
			H^1(U_2, \Omega^2_{U_2}(\log D_2 + E_2)) \ar[d]^{\iota_2^*} \ar[r]^{\iota^*} & 
			H^1(U', \Omega^2_{U'}(\log D')) \\
		H^1(U'_2, \Omega^2_{U'_2}(\log D'_2 + E'_2)) \ar[d]^{\phi_{G'_2}} &  \\
			H^1(U'_2,  \Omega^2_{U'_2}(\log D'_2 + E'_2)(G'_2)) \ar[r]^{\iota_{12}^*} &  
			H^1(U''_1,  \Omega^2_{U''_1}(\log D''_1 + E''_1)(G''_1)), \ar[uu]^{\iota_1^*}
			}
			\end{equation}
			where $\iota^*, \iota_1^*, \iota_2^*, \iota_{12}^*$ are the restrictions by open immersions 
			$\iota, \iota_1, 
			\iota_2, \iota_{12}$ as in the diagram (\ref{openimmdiagram}) and 
			$\phi_{G'_2}$ is induced by an injection $\mathcal{O}_{U_2'} \hookrightarrow \mathcal{O}_{U'_2}(G'_2)$.		Thus we have  
			\begin{equation}
			\Ker \phi_U = \Image \iota^* \subset \Image \iota_1^* \circ \iota_{12}^*.  
			\end{equation} 
			Now we shall prove the statement (i). 
			
	\noindent (i)	First, we prepare the diagram (\ref{iota1iota12diag}). 
		Since $U'_2$ is smooth and $D_2'+E_2'$ is a SNC divisor by the construction of $\mu_{12}$ in Lemma \ref{U_2lem}, we have a natural isomorphism 
		\[
		T^1_{(U_2', D_2' +E_2')} \simeq H^1(U_2', \Theta_{U_2'}(-\log D_2' +E_2')) \simeq H^1(U_2', \Omega^2_{U'_2}(\log D_2' +E_2')(-K_{U'_2}- D_2'- E_2')). 
		\]  
		The isomorphism $S_D \colon \mathcal{O}_U(-K_U -D) \simeq \mathcal{O}_U$ 
		induces an isomorphism $\mu_2^*(S_D) \colon \mathcal{O}_{U_2}(\mu_2^*(-K_U - D)) \simeq \mathcal{O}_{U_2}$ 
		and this induces an isomorphism 
		\[
		H^1(U_2', \Omega^2_{U'_2}(\log D_2' +E_2')(-K_{U'_2}- D_2'- E_2')) \simeq H^1(U'_2,  \Omega^2_{U'_2}(\log D'_2 + E'_2)(G'_2)). 
		\]
		Thus we have an isomorphism 
		\[
		\varphi_{\mu_2^*(S_D)} \colon T^1_{(U_2', D_2' +E_2')} \simeq H^1(U'_2,  \Omega^2_{U'_2}(\log D'_2 + E'_2)(G'_2)). 
		\]
		The homomorphisms $\varphi_{S_D}$ in (\ref{varphiSDdefn}), $\varphi_{\mu_2^*(S_D)}$ and $\iota_1^* \circ \iota_{12}^*$ fit in the commutative diagram 
		\begin{equation}\label{iota1iota12diag}
			\xymatrix{
			H^1(U'_2,  \Omega^2_{U'_2}(\log D'_2 + E'_2)(G'_2)) \ar[r]^{\ \ \ \ \ \ \iota_1^* \circ \iota_{12}^*}  
			& H^1(U', \Omega^2_{U'}(\log D'))  \\ 
			T^1_{(U_2', D_2' +E'_2)} \ar[r] \ar[d]^{(\mu'_{12})_*} \ar[u]_{\varphi_{\mu_2^*(S_D)}}^{\simeq} &
			 T^1_{(U',D')} \ar[u]^{\simeq} \\
			 T^1_{(U_1', D_1' +E_1')} \ar[d]^{\simeq} & T^1_{(U,D)} \ar[u]^{\simeq} \ar@/_2pc/[uu]_{\varphi_{S_D}} \\ 
			 T^1_{(U_1, D_1 +E_1)}, \ar[ru]^{(\mu_1)_*} &  
			}
			\end{equation}  
			where $(\mu'_{12})_*$ is the blow-down morphism as in Proposition \ref{pairblowdownprop}.

			By this, the above diagram and the previous relations, we see the relation
			\[
			\Ker \phi_U = \Image \iota^* \subset \Image \iota_1^* \circ \iota_{12}^* \subset \varphi_{S_D} (\Image (\mu_1)_*).
			\] 
			Thus we obtain the relation (\ref{eq:lemlogkerm^2}).  
			
\noindent(ii) Now assume that $U$ has only quotient singularity. 	
 By Lemma \ref{blowdownlem} (ii) and (\ref{eq:lemlogkerm^2}), we obtain the claim. 
			Thus we finish the proof of Lemma \ref{lem:logkerm^2}. 
		\end{proof}

\subsection{Proof of Theorem} 

We define the ``V-smooth pair'' as follows. 

\begin{defn}\label{Steenbrinkdefn}
	Let $U$ be a 3-fold with only terminal quotient singularities and $D \subset U$ its reduced divisor. 
	A pair $(U,D)$ is called a {\it V-smooth pair} if, for each point $p$, 
	there exists a Stein neighborhood $U_p$ 
	such that the index one cover $ \pi_p \colon V_p \rightarrow U_p$ satisfies that $\pi_p^{-1}(D \cap U_p) \subset V_p$ is a smooth divisor. 
	\end{defn}

We define ``simultaneous $\mathbb{Q}$-smoothing'' as follows. 

\begin{defn}\label{simultaneousQsmoothingdefn}
Let $X$ be a $3$-fold with only terminal singularities and 
$D \in |{-}K_X|$ an anticanonical element. 

We call a deformation $f \colon (\mathcal{X}, \mathcal{D}) \rightarrow \Delta^1$ of $(X, D)$ 
 a {\it simultaneous $\mathbb{Q}$-smoothing} if $\mathcal{X}_t$ and $\mathcal{D}_t$ 
  have only quotient singularities and 
  $(\mathcal{X}_t, \mathcal{D}_t)$ is a V-smooth pair (Definition \ref{Steenbrinkdefn}).  	
	\end{defn}

We give the proof of the main theorem in the following.

\begin{thm}\label{maintheorem}
Let $X$ be a $\mathbb{Q}$-Fano $3$-fold  
such that there exists an element $D \in |{-}K_X|$  with only isolated singularities. 

Then $(X,D)$ has a simultaneous $\mathbb{Q}$-smoothing. 
\end{thm}

\begin{proof}
	By Theorem \ref{qsmqfanothmintro}, we can assume that $X$ has only quotient singularities and $A_{1,2}/4$-singularities. 
We can also assume that $D$ has at worst Du Val singularities outside the non-Gorenstein points of $X$ by Theorem 
\ref{thm:takagielephant}. 
	
	Let $m$ be a sufficiently large integer such that $|{-}mK_X|$ contains a smooth element $D_m$ such that $\Sing D \cap D_m = \emptyset$. 
	Let $\pi \colon Y \rightarrow X$ be a cyclic cover branched along $D_m$ and $\Delta:= \pi^{-1}(D)$. 
	This induces an index one cover around each points of $\Sing X$ and $Y$ has only $A_{1,2}$-singularities, where 
	an $A_{1,2}$-singularity is a singularity analytically isomorphic to $0 \in (x^2+y^2+z^3+u^2=0) \subset \mathbb{C}^4$.  
	
	Let $ p_1, \ldots, p_l \in \Sing D$ be the image of non-Du Val singular points of $\Delta$ and $p_{l+1}, \ldots, p_{l+l'}$ 
	the image of Du Val singularities of $\Delta$. 
	Let $U_i \subset X$ be a Stein neighborhood of $p_i$ and $D_i:= D \cap U_i$ for $i=1, \ldots, l+l'$. 
	For $i=1,\ldots, l$, let $\mu_{i,1} \colon U_{i,1} \rightarrow U_i$ be the weighted blow-up constructed in Section \ref{firstblowupsection} and 
	$\mu_{i,12} \colon U_{i,2} \rightarrow U_{i,1}$ the birational morphism constructed in Lemma \ref{U_2lem}. 
	Let $\mu_{i,2}:= \mu_{i,1} \circ \mu_{i,12} \colon U_{i,2} \rightarrow U_i$ be the composition. 
	For $i=l+1, \ldots, l+l'$, let $\mu_i \colon \tilde{U}_i \rightarrow U_i$ be a projective birational morphism such that $\tilde{U}_i$ is smooth, $\mu_i^{-1}(D_i)$ is  
 a SNC divisor and $\mu_i$ is an isomorphism outside $p_i$.   
	
	By patching these $\mu_{i,2}$ for $i=1, \ldots, l$ and $\mu_i$ for $i=l+1, \ldots, l+l'$, 
	we construct a projective birational morphism $\mu \colon \tilde{X} \rightarrow X$ such that $\tilde{X}$ is smooth and
	 $\mu^{-1}(D) \subset \tilde{X}$ is a SNC divisor. 
	Let $\tilde{D} \subset \tilde{X}$ be the strict transform of $D$ and $E \subset \tilde{X}$ the $\mu$-exceptional divisor. 
	Also let $\tilde{D}_i := \tilde{D} \cap \mu^{-1}(U_i)$ and $E_i:= \mu^{-1}(p_i)$ for $i=1, \ldots, l+l'$. 
 
 We use the following diagram; 
 \begin{equation}\label{minadiagmult2}
	 \xymatrix{
	 H^1(X', \Omega_{X'}^2(\log D')) \ar[r]^{ \psi} \ar[d]^{\oplus p_{U_i}} & H^2_E(\tilde{X}, \Omega^2_{\tilde{X}}(\log \tilde{D} +E)) 
	 \ar[d]_{\oplus \varphi_i}^{\simeq} 
	 \ar[r] & H^2(\tilde{X}, \Omega^2_{\tilde{X}}(\log \tilde{D}+E)) \\
	 \oplus_{i=1}^{l+l'} H^1(U_i', \Omega^2_{U_i'}(\log D_i')) \ar[r]^{ \oplus \phi_i} & 
	 \oplus_{i=1}^{l+l'} H^2_{E_i}( \tilde{U}_i, \Omega^2_{\tilde{U}_i}(\log \tilde{D}_i +E_i)),  
	 &
	 }
	 \end{equation}
	where $X':= X \setminus \{p_1, \ldots, p_{l+l'} \}$ and $D' := D \cap X'$. 
	
 For $i = 1, \ldots, l $, let   
	 $\eta_i \in H^1(U'_i, \Omega^2_{U'_i}(\log D'_i))$ be an element inducing a simultaneous $\mathbb{Q}$-smoothing of $(U_i, D_i)$. 
	We see that $H^2(\tilde{X}, \Omega^2_{\tilde{X}}(\log \tilde{D}+E)) =0$ 
	since $\tilde{X} \setminus (\tilde{D}+E) \simeq X \setminus D$ is a smooth affine variety and 
	 $H^2(\tilde{X}, \Omega^2_{\tilde{X}}(\log \tilde{D}+E))$ is a subquotient of $H^4(\tilde{X} \setminus (\tilde{D}+E), \mathbb{C}) =0$
	 by the mixed Hodge theory on a smooth affine variety. 
	Thus there exists $\eta \in H^1(X', \Omega^2_{X'}(\log D'))$ such that $\psi_i(\eta) = (\varphi_i)^{-1}(\phi_i(\eta_i))$ for $i=1 ,\ldots, l$ and 
	$\psi_i(\eta) =0$ for $i=l+1, \ldots l+l'$. 
	
	Consider $1 \le i \le l$. 
	By $\eta_i - p_{U_i}(\eta) \in \Ker \phi_i$ and Lemma \ref{lem:logkerm^2}(i),  
	we obtain 
	\[\pr_{U_i} (\varphi_{S_{D_i}}^{-1} (\eta_i - p_{U_i}(\eta)))=0 \in T^1_{U_i} 
	\] 
	for $\varphi_{S_{D_i}} \colon T^1_{(U_i, D_i)} \stackrel{\sim}{\rightarrow} H^1(U'_i, \Omega^2_{U'_i}(\log D'_i))$. 
	Hence $p_{U_i}(\eta)$ induces a $\mathbb{Q}$-smoothing of $U_i$. 
	Thus it is enough to consider the case where $U_i$ has only quotient singularity. 
	In this case, we have   
	\begin{equation}\label{differm^2}
	\varphi_{S_{D_i}}^{-1} (\eta_i - p_{U_i}(\eta)) \in \mathfrak{m}^2 T^1_{(U_i,D_i)}
	\end{equation} 
	by Lemma \ref{lem:logkerm^2}(ii). 
	Let $\pi_i \colon V_i \rightarrow U_i$ be the index one cover and $\Delta_i:= \pi_i^{-1}(D_i) \subset V_i$. 
	By (\ref{differm^2}) and Lemma \ref{vistoli}, we see that $p_{U_i}(\eta)$ induces a smoothing of $\Delta_i$. Thus it induces a simultaneous $\mathbb{Q}$-smoothing of $(U_i, D_i)$ as well.  
	
	By \cite[Theorem 2.17]{Sano}, we can lift the first order deformation $\eta$ 
	to a deformation $f \colon (\mathcal{X}, \mathcal{D}) \rightarrow \Delta^1$ 
	of $(X,D)$ over a unit disc $\Delta^1$. 
	This $f$ induces a simultaneous $\mathbb{Q}$-smoothing of $(U_i,D_i)$ for $i=1,\ldots,l$. 
	Thus we can deform all non-Du Val singularities of $D$ and 
	obtain a $\mathbb{Q}$-Fano $3$-fold with a Du Val elephant as a general fiber of the deformation $f$.  
	Moreover, by \cite[Theorem 1.9]{Sano}, there exists a simultaneous $\mathbb{Q}$-smoothing of this $\mathbb{Q}$-Fano $3$-fold. 
	Thus we finish the proof of Theorem \ref{maintheorem}.   
	\end{proof}

\section{Examples}\label{examplesection}
Shokurov and Reid proved the following theorem. 

\begin{thm}
Let $X$ be a Fano $3$-fold with only canonical Gorenstein singularities. 

Then a general member $D \in |{-}K_X|$ has only Du Val singularities. 
\end{thm}

For non-Gorenstein $\mathbb{Q}$-Fano $3$-folds, this statement does not hold. 
We give several examples of $\mathbb{Q}$-Fano $3$-folds without Du Val elephants. 

\begin{eg}(\cite{Iano-Fletcher})
Iano-Flethcer gave an examples of a $\mathbb{Q}$-Fano $3$-fold without elephants.  
Let $X:= X_{12,14} \subset \mathbb{P}(2,3,4,5,6,7)$ be a weighted complete intersection of degree $12$ and $14$. 
Then we have $|{-}K_X| = \emptyset$ and general $X$ have only terminal quotient singularities. 
\end{eg}

Iano-Fletcher gave a list of 95 families of $\mathbb{Q}$-Fano $3$-fold weighted hypersurfaces. 
General members of those families have only quotient singularities and they have Du Val elephants.   
However, by taking special members in those families, we can construct weighted hypersurfaces without Du Val elephants as follows. 

\begin{eg}
Let $X:= X_{14}:= ((x^{14}+x^2 y_1^6 )+w^2+y_1^3 y_2^4 + y_2^7 + y_1z^4 =0) \subset \mathbb{P}(1,2,2,3,7)$ be 
a weighted hypersurface with coordinates $x, y_1, y_2, z, w$ of weights $1,2,2,3,7$ respectively. 
This is a modified version of an example in \cite[4.8.3]{ABR}.  

We can check that $X$ has only terminal singularities. 
It has three $1/2(1,1,1)$-singularities on the $(y_1,y_2)$-axis, a terminal singularity $(x^2+w^2+z^4 +y_2^4=0)/ \mathbb{Z}_2(1,1,1,0)$ and 
a $1/3(1,2,1)$-singularity at $[0:0:0:1:0]$.   

We see that $|{-}K_X| = \{D \}$ and $D$ has an elliptic singularity $(w^2+y_2^4 +z^4 =0) / \mathbb{Z}_2(1,0,1)$.  
In fact, this is log canonical. 
\end{eg}

\begin{eg}\label{egwhsnonlcelephant}
Let $X:= (x^{15} +x y^7 + z^5 + w_1^3 +w_2^3 =0) \subset \mathbb{P}(1,2,3,5,5)$ be a weighted hypersurface, 
where $x,y,z,w_1,w_2$ are coordinate functions with degrees $1,2,3,5,5$ respectively. 
We can check that $X$ has a $1/2(1,1,1)$-singularity and three $1/5(1,2,3)$-singularities. 
Thus $X$ is a $\mathbb{Q}$-Fano $3$-fold with only terminal quotient singularities. 

On the other hand, we have $|{-}K_X| = \{D \}$, where $D:= (z^5+w_1^3 +w_2^3=0) \subset \mathbb{P}(2,3,5,5)$. 
We see that the singularity $p=[1:0:0:0] \in D$ is isomorphic to a singularity $(x_1^5+x_2^3 +x_3^3=0)/ \mathbb{Z}_2$, where the $\mathbb{Z}_2$-action is of type 
$1/2(1,1,1)$. The singularity is not Du Val. 
This is also log canonical. 

We exhibit a simultaneous $\mathbb{Q}$-smoothing of this $(X,D)$ explicitly. 
For $\lambda \in \mathbb{C}$, let 
\[
X_{\lambda}:= (x^{15} +x y^7 + z^5 + w_1^3 +w_2^3 + \lambda y^6 z =0) \subset \mathbb{P}(1,2,3,5,5). 
\] 
For sufficiently small $\lambda \neq 0$, we see that 
$X_{\lambda}$ has only terminal quotient singularities and a Du Val elephant. 
Indeed, we see that $|{-}K_{X_{\lambda}}| = \{D_{\lambda} \}$, where 
\[
D_\lambda \simeq (z^5 +w_1^3 +w_2^3 + \lambda y^6 z =0) \subset \mathbb{P}(2,3,5,5)
\]
is a quasi-smooth hypersurface with only Du Val singularities.  
\end{eg}




\begin{eg}
Let $X:= X_{16} := (x^{16} +x(z^5 + zy^6) + yu^2+w^4 =0) \subset \mathbb{P}(1,2,3,4,7)$ be a weighted hypersurface 
with coordinates $x,y,z,w,u$ with weights $1,2,3,4,7$ respectively. 

Firstly, we check that $X$ has only terminal singularities. 
By computing the Jacobian of the defining equation of $X$, we see that $X$ is quasi-smooth outside the points on an affine piece $y \neq 0$ 
such that $x=w=u=0$ and $z(z^4 +y^6 =0)$. We can describe the singularities as follows; 
An affine piece $(x \neq 0)$ is smooth. An affine piece $(y \neq 0)$ has two singularities isomorphic to $(xz +w^4 +u^2=0) \subset \mathbb{C}^4$ and
 an singularity $(xz +w^4+u^2=0)/\mathbb{Z}_2$, where $\mathbb{Z}_2$ acts on $x,z,w,u$ with weights $1/2(1,1,0,1)$. 
 They are terminal by the classification (\cite[Theorem 6.5]{KSB}). 
 On a piece $(z \neq 0)$, there exists a $1/3(2,1,2)$-singularity. 
 A piece $(w \neq 0)$ is smooth. A piece $(u \neq 0)$ has a $1/7(1,3,4)$-singularity.   

Next, we check that $|{-}K_X|$ has only non-normal elements. 
Indeed, we have $|{-}K_X| = \{D \}$ with $D= (yu^2 + w^4=0) \subset \mathbb{P}(2,3,4,7)$ and the singular locus $\Sing D$ is non-isolated. 
Actually, $D$ is not normal crossing in codimension $1$. 
We also see that $\Sing D \simeq \mathbb{P}^1 \sqcup \{ {\rm pt} \}$. 

\end{eg}

We could not find an example of a $\mathbb{Q}$-Fano $3$-fold without Du Val elephants such that $h^0(X,-K_X) \ge 2$. 
Thus the following question is natural. 

\begin{prob}
Let $X$ be a $\mathbb{Q}$-Fano $3$-fold such that $h^0(X, -K_X) \ge 2$. 

Does there exist a Du Val elephant of $X$? Or, does there exist a normal elephant of $X$? 
\end{prob}

We can find an example of a klt $\mathbb{Q}$-Fano $3$-fold with only isolated quotient singularities whose anticanonical 
system contains only non-normal elements as follows. 

\begin{eg}\label{kltnonnormalexample}
	Let $X:=X_{15} \subset \mathbb{P}(1,1,5,5,7)$ be a general weighted hypersurface of degree $15$ in the weighted projective space.  
	Then $X$ has only three $1/5(1,1,2)$-singularities and one $1/7(1,5,5)$-singularity. 
	We see that $-K_X = \mathcal{O}_X(4)$ and the linear system $|{-}K_X|$ contains only reducible members. 
	Since general hypersurfaces $X$ satisfy this property, the statement as in Conjecture \ref{simulqsmconj} (ii) does not hold in this case. 
	\end{eg}

\section{Non-isolated case}

If every member of  $|{-}K_X|$ has non-isolated singularities, the deformation of singularities gets complicated and we do not know the answer for Conjecture \ref{simulqsmconj}. 
 However, we can reduce the problem to certain local setting as follows. 

\begin{thm}\label{nonisolsimulqsm}
	Let $X$ be a $\mathbb{Q}$-Fano $3$-fold. Assume that there exists a reduced member $D \in |{-}K_X|$ such that $C:= \Sing D$ is non-isolated. 
	Let $U_C$ be an analytic neighborhood of $C$ and $D_C := D \cap U_C$. 
	Assume also that there exists a deformation $(\mathcal{U}_C, \mathcal{D}_C) \rightarrow \Delta^1$ 
	such that $\mathcal{D}_{C,t}$ has only isolated singularities for $0 \neq t \in \Delta^1$. 
	
	Then there exists a simultaneous $\mathbb{Q}$-smoothing of $(X,D)$.  
	\end{thm}

For the proof of the theorem, we need to construct the following resolution of singularities of $(X,D)$. 
The construction is similar to Lemma \ref{U_2lem}. 

\begin{prop}\label{essresolpropnonisol}
Let $X$ be a $3$-fold with only terminal singularities and $D$ be its reduced divisor 
whose singular locus $C:= \Sing D$ is non-isolated. 
There exists a projective birational morphism $\mu \colon \tilde{X} \rightarrow X$ and 
a $0$-dimensional subset $Z \subset X$ with the following properties;
\begin{enumerate}
	\item[(i)] $\tilde{X}$ is smooth and $\mu^{-1}(D)$ has SNC support. 
	\item[(ii)] $\mu$ is an isomorphism over $X \setminus \Sing D \cup \Sing X$. 
	\item[(iii)] $\mu' \colon \tilde{X}' := \mu^{-1}(X \setminus Z) \rightarrow X':= X \setminus Z$ can be written as 
	a composition 
	\[
	\mu': \tilde{X}' = X'_{k} \stackrel{\mu'_{k-1}}{\rightarrow} X'_{k-1} \rightarrow \cdots \rightarrow X'_2 \stackrel{\mu'_1}{\rightarrow} 
	X'_1 = X', 
	\]
	where $\mu'_i \colon X'_{i+1} \rightarrow X'_i$ is an isomorphism or a blow-up of a smooth curve $Z'_i$ with either of the following;  
	\begin{itemize}
	\item If the strict transform $D'_i \subset X'_i$ of $D':= D \cap X' \subset X'$ is singular, we have $Z'_i \subset \Sing D'_i$.  
	\item If $D'_i$ is smooth, we have $Z'_i \subset D'_i \cap E'_i$, where $E'_i$ is the exceptional divisor of 
	$\mu'_{i,1}:= \mu'_1 \circ \cdots \circ \mu'_{i-1} \colon X'_i \rightarrow X'_1$. 
	\end{itemize}
	\end{enumerate}	
	
	As a consequence, the divisor 
	\begin{equation}\label{eq:logcanoantieff}
	-(K_{X'_{k}} + D'_k + E'_k) + (\mu')^* (K_{X'} +D')  
	\end{equation}
	is an effective divisor supported on $E'_k$. 
	\end{prop}

\begin{proof}
	Let $\nu_1 \colon X_1 \rightarrow X$ be a composition of blow-ups of smooth centers such that $X_1$ is smooth, 
	the exceptional locus $E_1$ of $\nu_1$ is a SNC divisor and 
	$\nu_1$ is an isomorphism over $X \setminus \Sing X$. 
	Thus the strict transform $D_1 \subset X_1$ of $D$ is a reduced Cartier divisor. 
	   By applying \cite[Theorem A.1]{BierstoneMilmanMinimal1} to the pair $(X_1, D_1)$, we can construct a composition of blow-ups
	 \[
	 \mu_{k,1} \colon X_{k} \stackrel{\mu_{k-1}}{\rightarrow} \cdots \rightarrow X_2 \stackrel{\mu_1}{\rightarrow} X_1,  
	 \]
	 where $\mu_i \colon X_{i+1} \rightarrow X_i$ is a blow-up of a smooth center $Z_i \subset X_i$ such that, 
	 for each $i$, 
	 \begin{itemize}
		 \item $D_i \subset X_i$ is the strict transform of $D_1$, 
		 \item $E_i :=\mu_{i,1}^{-1}(E_1) \subset X_i$ is the exceptional divisor, 
		 where $\mu_{i,1} := \mu_{1} \circ \cdots \circ \mu_{i-1} \colon X_i \rightarrow X_1$, 
		 \end{itemize} 
		 then, for each $i$, 
		 \begin{enumerate}
			 \item[(i)'] $Z_i$ and $E_i$ intersect transversely, 
			 \item[(ii)'] $Z_i \subset \Sing D_i$ or $D_i$ is smooth and $Z_i \subset D_i \cap E_i$, 
			 \item[(iii)'] $X_k$ is smooth and $D_k \cup E_k$ is a SNC divisor. 
			 \end{enumerate}
			 
			 Let $\tilde{X}:= X_k$, $\mu:= \nu_1 \circ \mu_{k,1} \colon \tilde{X} \rightarrow X$ and 
			 \[
			 Z:= \Sing X \cup \bigcup_{\dim \nu_1 (\mu_{i,1}(Z_i)) =0} \nu_1(\mu_{i,1}(Z_i)) \subset X
			 \]
			  the union of $0$-dimensional images of the centers on $X_1$ and the singular locus $\Sing X$ of $X$.  
			 Then we see that these $\tilde{X}, \mu, Z$ satisfy the condition (i) in the statement by the construction of $\mu$. 
			 We can check (ii) by (ii)'. 
			 We check (iii) as follows. 
			
			  Let $X':= X \setminus Z$, $\tilde{X}':= \mu^{-1}(X')$ and 
			  $\mu':= \mu|_{\tilde{X}'} \colon \tilde{X}' \rightarrow X'$. 
			  Let $X'_i:= \mu_{i,1}^{-1}(X')$, $D'_i:= D_i \cap X'_i$ and $E'_i := E_i \cap X'_i$ as well. 
			  We see that $\mu'$ is a composition of blow-ups of smooth curves $Z'_i := Z_i \cap X'_i$ 
			  with the property (iii) in the statement by the property (ii)' of $\mu_{k,1}$. 
			 
			 We can check the last statement about (\ref{eq:logcanoantieff}) as follows. 
			 For $j \le i$, let $\mu'_{i,j} := \mu_j' \circ \cdots \circ \mu'_{i-1} \colon X'_i \rightarrow X'_j$. 
			 We have an equality 
			 \begin{multline}\label{sumequality}
			 -(K_{X'_{k}} + D'_k + E'_k) + (\mu')^* (K_{X'} +D') \\ 
			 = -E'_k + 
			 \sum_{i=1}^{k-1} (\mu'_{k, i+1})^* (-(K_{X'_{i+1}} + D'_{i+1}) + (\mu'_i)^*(K_{X'_i} +D'_i)).   
			  \end{multline}
			  By the condition (iii) of the resolution $\mu$ in the statement, we see that the divisor  
			\[
			-(K_{X'_{i+1}} + D'_{i+1}) + (\mu'_i)^*(K_{X'_i} +D'_i) = (\mult_{Z'_i}(D'_i) -1) (\mu'_i)^{-1}(Z'_i). 
			\]
			 is effective. Moreover, for $i_0:= \min \{i \mid Z'_i \neq \emptyset \}$, we see that $\mult_{Z'_{i_0}} (D'_{i_0}) -1 >0$ and 
			 \[
			  (\mu'_{k, i_0+1})^* ((\mu'_{i_0})^{-1}(Z'_{i_0})) \ge E'_k 
			 \]
			since $Z'_{i_0} \subset \Sing D'_{i_0}$ and $Z'_i$ is contained in the $\mu'_{i,1}$-exceptional divisor  for all $i$. 
Hence, by the equality (\ref{sumequality}), we obtain the effectivity of $-(K_{X'_{k}} + D'_k + E'_k) + (\mu')^* (K_{X'} +D')$. 

Thus we finish the proof of Proposition \ref{essresolpropnonisol}.
	\end{proof}

 We shall use the above resolution $\mu \colon \tilde{X} \rightarrow X$ of the pair $(X,D)$ and use the same notations in the following. 
 
 Let $\tilde{D} \subset \tilde{X}$ be the strict transform of $D$ and $E:= \Exc \mu$ be the exceptional divisor.
 Let $\tilde{D}':= \tilde{D} \cap \tilde{X}'$ and $E':= \Exc \mu'$ for $\mu' \colon \tilde{X}' \rightarrow X'$. 
 By Proposition \ref{essresolpropnonisol}, we see the linear equivalence  
 \begin{equation}\label{eq:logcanoG'defn}
 -(K_{\tilde{X}'} +\tilde{D}' +E') + (\mu')^*(K_{X'} +D') \sim G'    
 \end{equation} 
for some effective divisor $G'$ supported on $\Exc \mu'$. 

Let $X'':= X \setminus \Sing D$ and $D'' := D \cap X''$. 
Let $\tilde{U}_C:= \mu^{-1}(U_C)$, $U_C'':= U_C \setminus \Sing D_C$, 
$\tilde{D}_C:= \tilde{D} \cap \tilde{U}_C$ and 
$D_C'':= D \cap U_C''$. Let $E_C \subset \tilde{U}_C$ be the exceptional divisor for 
$\mu_C \colon \tilde{U}_C \rightarrow U_C$. 
By the property (ii) in Proposition \ref{essresolpropnonisol}, we have open immersions 
$\tilde{\iota} \colon X'' \hookrightarrow \tilde{X}$ and $\tilde{\iota}_C \colon U_C'' \hookrightarrow \tilde{U}_C$. 
We consider the following diagram 
\begin{equation}\label{minadiagnonisol}
\xymatrix{
H^1(X'', \Omega^2_{X''}(\log D'')) \ar[r]^{\psi} \ar[d]^{\iota_C^*} 
& H^2_E(\tilde{X}, \Omega^2_{\tilde{X}}(\log \tilde{D}+E)) \ar[r] 
\ar[d]_{\simeq}^{\pi_C}
& H^2(\tilde{X}, \Omega^2_{\tilde{X}}( \log \tilde{D}+E)) \\ 
H^1(U_C'', \Omega^2_{U_C''}(\log D_C'')) \ar[r]^{\phi_C \ \ \ \ } & H^2_{E}(\tilde{U}_C, \Omega^2_{\tilde{U}_C}( \log \tilde{D}_C + E_C)), & 
}
\end{equation}
where the homomorphisms $\psi$ and $\phi_C$ are the coboundary maps and 
the homomorphism $\iota_C^*$ is a restriction by an open immersion $\iota_C \colon U_C \hookrightarrow X$. 

Let $p \in C \setminus Z$ and $U_p \subset X$ a Stein neighborhood of $p$. 
Let $\tilde{U}_p:= \mu^{-1}(U_p)$, $\mu_p := \mu|_{\tilde{U}_p} \colon \tilde{U}_p \rightarrow U_p$, 
$D_p:= D \cap U_p$, 
$U''_p:= U_p \setminus \Sing D_p$ and $D''_p:= D_p \cap U''_p$. 
We also have an open immersion $U''_p \hookrightarrow \tilde{U}_p$.  
Hence the coboundary map $\phi_C$ fits in the following commutative diagram;  
\begin{equation}\label{phiCphipdiag}
\xymatrix{
H^1(U''_C, \Omega^2_{U''_C}(\log D''_C)) \ar[r]^{\phi_C \ \ \ \ } \ar[d]^{\iota_{C,p}^*} & H^2_{E_C}(\tilde{U}_C, \Omega^2_{\tilde{U}_C} (\log \tilde{D}_C + E_C)) 
\ar[d]^{\iota_{C,p}^*} \\
H^1(U''_p, \Omega^2_{U''_p}(\log D''_p)) \ar[r]^{\phi_p \ \ \ \ } & H^2_{E_p}(\tilde{U}_p, \Omega^2_{\tilde{U}_p}(\log \tilde{D}_p +E_p)), 
}
\end{equation}
where the horizontal maps are coboundary maps of local cohomology and the vertical maps are induced by 
the open immersion $\iota_{C,p} \colon U_p \hookrightarrow U_C$. 

Fix an isomorphism $\varphi_{(U_p, D_p)} \colon \mathcal{O}_{U_p} \simeq \mathcal{O}_{U_p}(-K_{U_p} - D_p)$. 
This induces isomorphisms 
\[
T^1_{(U''_p, D''_p)} \simeq H^1(U''_p, \Theta_{U''_p} (- \log D''_p)) \stackrel{\sim}{\rightarrow} H^1(U''_p, \Omega^2_{U''_p}(\log D''_p)),  
\] 
\[
T^1_{(\tilde{U}_p, \tilde{D}_p +E_p)} \simeq H^1(\tilde{U}_p, \Theta_{\tilde{U}_p} (- \log \tilde{D}_p +E_p)) 
\stackrel{\sim}{\rightarrow} H^1(\tilde{U}_p, \Omega^2_{\tilde{U}_p}(\log \tilde{D}_p +E_p)(G_p)), 
\]
where we set $G_p := G'|_{\tilde{U}_p}$ for $G'$ in (\ref{eq:logcanoG'defn}). 
These isomorphisms fit in the commutative diagram 
\begin{equation}\label{iotaptildediag}
\xymatrix{
T^1_{(\tilde{U}_p, \tilde{D}_p +E_p)} \ar[r]^{(\tilde{\iota}_p)^*} \ar[d]^{\simeq} & T^1_{(U''_p, D''_p)} \ar[d]^{\simeq} \\
H^1(\tilde{U}_p, \Omega^2_{\tilde{U}_p}(\log \tilde{D}_p +E_p)(G_p)) \ar[r]^{\ \ (\tilde{\iota}_p)^*} & H^1(U''_p, \Omega^2_{U''_p}(\log D''_p)) 
}
\end{equation}
and we use the same symbol $(\tilde{\iota}_p)^*$ for the both horizontal maps.

We have the following lemma. 
\begin{lem}\label{kerimlemnonisol}
	We have a relation 
	\[
	\Ker \phi_p \subset \Image (\tilde{\iota}_p)^* \subset H^1(U''_p, \Omega^2_{U''_p}(\log D''_p)),
	\]
	 where $\phi_p$ and $(\tilde{\iota}_p)^*$ 
	are the homomorphisms in the diagrams (\ref{phiCphipdiag}) and (\ref{iotaptildediag}) respectively.  
	\end{lem}

\begin{proof}
	Since we have an exact sequence 
	\[
	H^1(\tilde{U}_p, \Omega^2_{\tilde{U}_p}(\log \tilde{D}_p +E_p)) \xrightarrow[]{\alpha_p}
	 H^1(U''_p, \Omega^2_{U''_p}(\log D''_p)) \xrightarrow[]{\phi_p}
	 H^2_{E_p}(\tilde{U}_p, \Omega^2_{\tilde{U}_p}(\log \tilde{D}_p +E_p)), 
	\]
	we obtain that $\Ker \phi_p = \Image \alpha_p$. 
	By this and the commutative diagram 
	\[
	\xymatrix{
	H^1(\tilde{U}_p, \Omega^2_{\tilde{U}_p}(\log \tilde{D}_p +E_p)) \ar[r]^{\alpha_p} \ar[d] 
	& H^1(U''_p, \Omega^2_{U''_p}(\log D''_p)) \\
H^1(\tilde{U}_p, \Omega^2_{\tilde{U}_p}(\log \tilde{D}_p + E_p) (G_p)) \ar[ru]_{(\tilde{\iota}_p)^*},  & 
} 
	\]
we obtain the claim. 	
	\end{proof}

The open immersion $\iota_p \colon U''_p \hookrightarrow U_p$ induces a restriction homomorphism 
$\iota_p^* \colon T^1_{(U_p, D_p)} \rightarrow T^1_{(U''_p, D''_p)}$. 
This is injective.  
Indeed, for $(\mathcal{U}_p, \mathcal{D}_p) \in T^1_{(U_p, D_p)}$, we see that 
\[
(\iota_p)_* \iota_p^* \mathcal{O}_{\mathcal{U}_p} \simeq \mathcal{O}_{\mathcal{U}_p}, \ \ 
(\iota_p)_* \iota_p^* \mathcal{I}_{\mathcal{D}_p} \simeq \mathcal{I}_{\mathcal{D}_p}
\]
since $U_p \setminus U''_p \subset U_p$ has codimension $2$, the divisor $D_p \subset U_p$ is Cartier and 
$U_p$ is ${\rm S}_2$.  
The open immersion $\tilde{\iota}_p \colon U''_p \hookrightarrow \tilde{U}_p$ also induces 
a restriction homomorphism 
$(\tilde{\iota}_p)^* \colon T^1_{(\tilde{U}_p, \tilde{D}_p +E_p)}  \rightarrow T^1_{(U''_p, D''_p)}$.  
These fit in the following diagram; 
\[
\xymatrix{
T^1_{(\tilde{U}_p, \tilde{D}_p +E_p)} \ar[r]^{(\tilde{\iota}_p)^*} \ar[rd]^{(\mu_p)_*} & T^1_{(U''_p, D''_p)} \\ 
 & T^1_{(U_p, D_p)} \ar[u]^{\iota_p^*}}, 
\]
where $(\mu_p)_*$ is the blow-down homomorphism as in Proposition \ref{pairblowdownprop}. 

Since $\iota_p^*$ is injective, we can regard $T^1_{(U_p, D_p)} \subset T^1_{(U''_p, D''_p)}$ and 
we obtain the relation 
\[
\Image (\tilde{\iota}_p)^* = \Image (\mu_p)_*. 
\]

Let $f_p \in \mathcal{O}_{U_p, p}$ be the defining equation of $D_p \subset U_p$. 
We have a description 
\[
T^1_{(U_p, D_p)} \simeq \mathcal{O}_{U_p, p}/ J_{f_p}, 
\]
where $J_{f_p} \subset \mathcal{O}_{U_p,p}$ is the Jacobian ideal determined by $f_p$. 

By the following lemma, we see that elements of $\Image (\mu_p)_*$ 
is induced by functions with orders $2$ or higher.

\begin{lem}\label{nonisolimagem^2lem}
	We have 
$\Image (\tilde{\iota}_p)^* = \Image (\mu_p)_* \subset \mathfrak{m}_p^2 T^1_{(U_p,D_p)}$. 
\end{lem}

\begin{proof}
	By Proposition \ref{essresolpropnonisol} (iii), 
	we see that $\mu_p \colon \tilde{U}_p \rightarrow U_p$ 
	is a composition of blow-ups 
	\[
	\tilde{U}_p = U_{k,p} \xrightarrow{\mu_{k-1,p}} U_{k-1,p} \rightarrow \cdots \rightarrow U_{1,p} \xrightarrow{\mu_{0,p}} U_{0,p} = U_p, 
	\]
	where $\mu_{i,p} \colon U_{i+1,p} \rightarrow U_{i,p}$ is a blow-up of 
	a smooth curve $C_{i,p}$ for $i=0, \ldots, k-1$. 
	Since we have $\Image (\mu_p)_* \subset \Image (\mu_{0,p})_*$, 
	it is enough to show that 
	\[
	\Image (\mu_{0,p})_* \subset  \mathfrak{m}_p^2 T^1_{(U_p,D_p)}. 
	\]
	
	Let $D_{1,p} \subset U_{1,p}$ be the strict transform of $D_p \subset U_p$ and 
	$E_{1,p} \subset U_{1,p}$ be the $\mu_{0,p}$-exceptional divisor.
	Let $\tilde{\eta}_p \in T^1_{(U_{1,p}, D_{1,p}+E_{1,p})}$
	and $C_p := C_{0,p}$. 
	Let $E^{[1]}_{1,p}$ be the first order deformation of $E_{1,p}$ induced by 
	$\tilde{\eta}_p$. 
	By taking the push-forward of the ideal sheaf of $E_{1,p}^{[1]} \subset U_{1,p}^{[1]}$, 
	 we obtain  
	a first order deformation $C^{[1]}_p$ of $C_p$.  
	Let $\eta_p := (\mu_{0,p})_* (\tilde{\eta}_p) \in T^1_{(U_p, D_p)}$ which 
	induces a deformation $(U^{[1]}_p, D^{[1]}_p)$ of $(U_{p}, D_{p})$ over $A_1:= 
	\mathbb{C}[t]/(t^2)$. This can be lifted to a deformation 
	$(\mathcal{U}_p, \mathcal{D}_p)$ of $(U_p, D_p)$ over $\Delta^1$ such that 
	$\mathcal{U}_p \simeq U_p \times \Delta^1$. 
	We can choose $\mathcal{D}_p$ so that $\mathcal{D}_p$ contains a deformation of $C_p$ 
	as follows.
	
	Since $C_{p}$ can be written as $C_p = (x_p = y_p =0) \subset U_p$ for some regular equations 
	$x_p, y_p$ on $U_p$, 
		we see that the deformation $C^{[1]}_p$ can be extended to  
		a deformation $\mathcal{C}_p$ of $C_p$ over $\Delta^1$. 
We have 
\begin{equation}\label{eqn:pullbackD_p}
\mu_p^* D_p = D_{1,p} + m_1 E_{1,p}
\end{equation} 
for some positive integer $m_1 \ge 2$  
	since $D_p$ is singular along $C_p$ by the property (iii) in Proposition \ref{essresolpropnonisol}.  
		 Note that $C^{[1]}_p \subset D^{[1]}_p$ since we construct $D^{[1]}_p$ by 
		\[
		(\mu_{0,p})_* \mathcal{O}_{U^{[1]}_{1,p}}(-D^{[1]}_{1,p} -m_1 E^{[1]}_{1,p}) = 
		\mathcal{O}_{U^{[1]}_p}(-D^{[1]}_p).  
		\]
	Thus we can choose a lifting $\mathcal{D}_p$ of $D_p^{[1]}$ 
	such that $\mathcal{C}_p \subset \mathcal{D}_p$.	 
		
		Let $\nu_p \colon \mathcal{U}_{1,p} \rightarrow \mathcal{U}_p$ 
		be the blow-up of $\mathcal{C}_p$ 
		and $\mathcal{D}_{1,p} \subset \mathcal{U}_{1,p}$ 
		be the strict transform of $\mathcal{D}_p \subset \mathcal{U}_p$. 
		Let $\mathcal{E}_{1,p} \subset \mathcal{U}_{1,p}$ be the $\nu_p$-exceptional divisor. 
		We see that 
	$(\mathcal{U}_{1,p}, \mathcal{D}_{1,p} + \mathcal{E}_{1.p})$ is a deformation of $(U_{1,p}, D_{1,p} +E_{1,p})$ over $\Delta^1$. 
	We have 
	\[
	\nu_p^* \mathcal{D}_p = \mathcal{D}_{1,p} + m_1 \mathcal{E}_{1,p} 
	\] 
	since, by restricting the above equality to $U_{1,p}$, we obtain (\ref{eqn:pullbackD_p}). 
	
	Thus we see that the fiber $\mathcal{D}_{p,t}$ of $\mathcal{D}_p \rightarrow \Delta^1$ 
	is singular along the fiber $\mathcal{C}_{p,t}$ of $\mathcal{C}_p \rightarrow \Delta^1$ over $t \in \Delta^1$. 
	Hence $\mathcal{D}_p$ should be induced by a function $h_p \in \mathfrak{m}_p^2$. 
	\end{proof}

As a summary of Lemma \ref{kerimlemnonisol} and \ref{nonisolimagem^2lem}, we obtain the relation 
\begin{equation}\label{summaryrelation}
\Ker \phi_p \subset \Image (\tilde{\iota}_p)^* \subset \mathfrak{m}_p^2 T^1_{(U_p,D_p)}. 
\end{equation}

\vspace{5mm}

By using these ingredients, we prove Theorem \ref{nonisolsimulqsm} in the following.

\begin{proof}[Proof of Theorem \ref{nonisolsimulqsm}] 
	We continue to use the same notations as above. 
	Let \[
	\eta_C \in H^1(U_C'', \Omega^2_{U_C''}(\log D_C'')) \simeq T^1_{(U_C'', D_C'')}
	\] be the element 
	which induces the deformation $(\mathcal{U}_C, \mathcal{D}_C) \rightarrow \Delta^1$ as in the assumption of Theorem \ref{nonisolsimulqsm}. 
	Let $p \in C \setminus Z$ and $\iota_{C,p} \colon U_p \hookrightarrow U_C$ an open immersion and 
	consider the element $\iota_{C,p}^* (\eta_C) \in H^1(U_p'', \Omega^2_{U_p''}(\log D_p''))$, 
	where the homomorphism $\iota_{C,p}^*$ is the one appeared in the diagram (\ref{phiCphipdiag}). 
	Note that $\iota_{C,p}^* (\eta_C)$ induces a smoothing of $D_p := D \cap U_p$. 
	By this and the relation (\ref{summaryrelation}), we see that $\iota_{C,p}^* (\eta_C) \notin \Ker \phi_p$. 
Note that $H^2(\tilde{X}, \Omega^2_{\tilde{X}}(\log \tilde{D}+E)) =0$ by the mixed Hodge theory on an open variety as in the proof of Theorem \ref{maintheorem} 
since $\tilde{X} \setminus (\tilde{D} \cup E) \simeq X \setminus D$ is affine. 
Hence, by the diagram (\ref{minadiagnonisol}),  there exists $\eta \in H^1(X'', \Omega^2_{X''}(\log D''))$ such that 
$\varphi_C^{-1}(\phi_C(\eta_C)) = \psi( \eta)$. 
We see that 
\begin{equation}\label{etanotinm^2} 
\iota_{C,p}^* (\iota_C^*(\eta)) \notin \mathfrak{m}_p^2 T^1_{(U_p, D_p)}.
\end{equation} 
Indeed, we have $\iota_{C,p}^* (\iota_C^*(\eta)-\eta_C) \in \Ker \phi_p \subset 
\Image (\mu_p)_* \subset \mathfrak{m}_p^2 T^1_{(U_p, D_p)}$ by Lemmas \ref{kerimlemnonisol} and \ref{nonisolimagem^2lem}. 
By the unobstructedness of deformations of $(X, D)$ \cite[Theorem 2.17]{Sano},  
we have a deformation $(\mathcal{X}, \mathcal{D}) \rightarrow \Delta^1$  of $(X,D)$ induced by $\eta$. 
   By (\ref{etanotinm^2}) and Lemma \ref{vistoli}, we see that $\mathcal{D}_t$ has only isolated singularities for $t \neq 0$. 
   Hence, by applying Theorem \ref{simulqsmisolthmintro} to $(\mathcal{X}_t, \mathcal{D}_t)$, 
   we finally obtain a simultaneous $\mathbb{Q}$-smoothing of $(X,D)$. 
	\end{proof}

\begin{rem}
	It is reasonable to assume the existence of a reduced elephant. 
	Actually, Alexeev proved that, if a $\mathbb{Q}$-Fano $3$-fold $X$ is $\mathbb{Q}$-factorial and its Picard number is $1$, then 
	there exists a reduced and irreducible elephant on $X$ (\cite[Theorem (2.18)]{AlexeevElephant}). 
	\end{rem}

\begin{rem}
	The assumption of Theorem \ref{nonisolsimulqsm} is satisfied if $|{-}K_{U_C}|$ contains a normal element. 
	For example, this happens if $C \simeq \mathbb{P}^1$ and it is contracted by some extremal contraction (\cite[(1.7)]{kollarmoriflip}). 
	\end{rem}


\section{Appendix: Existence of a good weighted blow-up}

Let $U= \mathbb{C}^3$ and $0 \in D \subset U$ a normal divisor with a non-Du Val singularity at $0 \in D$. 
As Lemmas \ref{nonGorwtlem} and \ref{A12-wtdblowup}, we can find a good weighted blow-up as follows. 
Although we do not need these results in this paper, we treat this for possible use for another problem. 

The following is an easiest case where a singularity on a divisor is a hypersurface singularity of multiplicity $3$ or higher.  

\begin{lem}\label{mult3U_1lemma}
	Let $U:= \mathbb{C}^3$ and  
	 $ D \subset U$ a divisor with an isolated singularity at $0$.  
	Assume that $m_D:= \mult_0 D \ge 3$. 
	Let $\mu_1 \colon U_1 \rightarrow U$ be the blow-up at the origin $0$
	and $E_1$ its exceptional divisor.

	Then the discrepancy $a(E_1, U, D)$ satisfies   
	\begin{equation}\label{statementmult3}
	a(E_1, U, D) = 2- m_D \le -1.  
	\end{equation} 
	\end{lem}
	 
\begin{proof}
	This follows since we have $K_{U_1} = \mu_1^* K_U +2 E_1$ and $D_1 = \mu_1^* D - m_D E_1$. 
	\end{proof}

We use the following notion of right equivalence (\cite[Definition 2.9]{GLSintro}).

\begin{defn}
	Let $\mathbb{C}\{x_1,\ldots, x_n \}$ be the convergent power series ring of $n$ variables. 
	Let $f,g \in \mathbb{C} \{x_1,\ldots, x_n \}$. 
	
	$f$ is called {\it right equivalent} to $g$ if there exists an automorphism $\varphi$ of 
	$\mathbb{C} \{x_1, \ldots , x_n \}$ such that $\varphi(f) = g$.  
	We write this as $f \overset{r}{\sim} g$. 
	\end{defn} 

The following double point in a smooth $3$-fold is actually the most tricky case. 

\begin{lem}\label{Gorwtbluplem}
	Let $0 \in D:= (f=0) \subset \mathbb{C}^3=: U$ be a divisor such that $\mult_0 D =2$ and $0 \in D $ is 
	not a Du Val singularity. 
	
	Then there exists a birational morphism $\mu_1 \colon U_1 \rightarrow U$ which is a weighted blow-up 
	of weights $(3,2,1)$ or $(2,1,1)$ for a suitable coordinate system on $U$ such that 
	the discrepancy $a(E_1, U, D)$ of the $\mu_1$-exceptional divisor $E_1$ satisfies 
	\[
	a(E_1, U, D) \le -1. 
	\]
	\end{lem}

\begin{proof}
	By taking a suitable coordinate change, we can write $f= x^2 +g(y,z)$ for some $g(y,z) \in \mathbb{C}[y,z]$ which defines 
	a reduced curve $(g(y,z) =0) \subset \mathbb{C}^2$.  
	We see that $\mult_0 g(y,z) \ge 3$ since, if $\mult_0 g(y,z) =2$, we see that $D$ has a Du Val singularity of type A at $0$.   
We can write $g(y,z) = \sum g_{i,j} y^i z^j$ for $g_{i,j} \in \mathbb{C}$. 
We divide the argument with respect to $\mult_0 g(y,z)$.

\vspace{5mm}

({\bf Case 1}) Consider the case $\mult_0 g(y,z) \ge 4$. 
Let $\mu_1 \colon U_1 \rightarrow U$ be the weighted blow-up with weights $(2,1,1)$ and $D_1 \subset U_1$  
the strict transform of $D$. 
Then we have 
\[
K_{U_1} = \mu_1^* K_U + 3 E_1, 
\]
\[
\mu_1^* D = D_1 + m_D E_1, 
\]
where $m_D = \min \{4, \min\{i+j \mid g_{i,j} \neq 0 \} \}$.  
By the assumption $\mult_0 g(y,z) \ge 4$, we see that $g_{i,j} \neq 0$ only if $i+j \ge 4$. 
Thus we see that $m_D = 4$.  
Thus we obtain  
\[
K_{U_1} +D_1 = \mu_1^* (K_U+D) - E_1  
\]
and the weighted blow-up $\mu_1$ satisfies the required property. 

\vspace{5mm}

({\bf Case 2}) Consider the case $\mult_0 g(y,z) =3$. 
Let $g^{(k)}:= \sum_{i+j\le k} g_{i,j} y^i z^j$ be the {\it $k$-jet} of $g$. 
We divide this into two cases with respect to $g^{(3)}$. 
The proof uses the arguments in the classification of simple singularities of type $D$ and $E$
(\cite[Theorem 2.51, 2.53]{GLSintro}). 

(2.1) Suppose that $g^{(3)}$ factors into at least two different factors. 
By \cite[Theorem 2.51]{GLSintro}, we see that $g \overset{r}{\sim} y(z^2 +y^{k-2})$ for some $k \ge 4$. 
Thus $0 \in D$ is a Du Val singularity of type $D_k$. 
This contradicts the assumption.  

(2.2) Suppose that $g^{(3)}$ has a unique linear factor. 
We can write $g^{(3)} = y^3$ by a suitable coordinate change. 
By the proof of \cite[Theorem 2.53]{GLSintro}, the $4$-jet $g^{(4)}$ can be written as 
\[
g^{(4)} = y^3 + \alpha z^4 + \beta yz^3
\]
for some $\alpha, \beta \in \mathbb{C}$. 

(i) If $\alpha \neq 0$, we obtain $g \overset{r}{\sim} y^3 +z^4$ by the same argument as \cite[Theorem 2.53, Case $E_6$]{GLSintro}. 
Thus we see that $0 \in D$ is a Du Val singularity of type $E_6$. 

(ii) If $\alpha =0$ and $\beta \neq 0$, we obtain $g \overset{r}{\sim} y^3 +yz^3$ by the same argument as \cite[Theorem 2.53, Case $E_7$]{GLSintro}. 
Thus we see that $0 \in D$ is a Du Val singularity of type $E_7$.

(iii) Now assume that $\alpha = \beta =0$. In this case, the $5$-jet $g^{(5)}$ can be written as 
\[
g^{(5)} = y^3 + \gamma z^5 + \delta y z^4  
\]
 for some $\gamma, \delta \in \mathbb{C}$. 

If $\gamma \neq 0$, we obtain $g \overset{r}{\sim} y^3 +z^5$ by the same argument as \cite[Theorem 2.53, Case $E_8$]{GLSintro}. 
Thus we see that $0 \in D$ is a Du Val singularity of type $E_8$. 

If $\gamma =0$ and $\delta \neq 0$, we can write $g = y^3 + yz^4 + h_6(y,z)$ for some $h_6(y,z) \in \mathbb{C}[y,z]$ such that 
$\mult_0 h_6(y,z) \ge 6$. 
Let $\mu_1 \colon U_1 \rightarrow U$ be the weighted blow-up with weights $(3,2,1)$ on $(x,y,z)$ 
and $E_1$ its exceptional divisor. 
Then we can calculate 
\[
K_{U_1} = \mu_1^*K_U + 5 E_1, 
\]
\[
\mu_1^* D = D_1 + 6 E_1    
\] 
by the formula  (\ref{stricttrfwtblup}).  
Thus we obtain 
\[
K_{U_1} +D_1 = \mu_1^*(K_U+D) -E_1. 
\]
Hence $\mu_1$ has the required property. 

If $\gamma= \delta =0$, we can write $g = y^3 + h_6$ for some nonzero $h_6$ such that $\mult_0 h(y,z) \ge 6$. 
Let $\mu_1 \colon U_1 \rightarrow U$ be the weighted blow-up with weights $(3,2,1)$ as above. 
We can similarly check that this $\mu_1$ has the required property. 

	\end{proof}

\section*{Acknowledgments}
The author would like to thank Professor Miles Reid for useful comments on the first blow-ups and examples. 
He acknowledges Professor Jonathan Wahl for letting him know the example of equisingular deformation. 
He acknowledges Professor Angelo Vistoli for letting him know Lemma \ref{vistoli}. 
Part of this paper is written during the author's stay in Princeton university. 
He would like to thank Professor J\'{a}nos Koll\'{a}r for useful comments on analytic neighborhoods and nice hospitality. 
Finally, he thanks the anonymous referee for reading the manuscript carefully and pointing out mistakes. 
He is partially supported by Warwick Postgraduate Research Scholarship, Max Planck Institut f\"{u}r Mathematik, JSPS fellowships for Young Scientists and JST Tenure Track Program.


\begin{thebibliography}{}

\bibitem{AlexeevElephant}
V.~Alexeev, 
\emph{General elephants of $\mathbb{Q}$-Fano $3$-folds},  
Compositio Math. 91 (1994), no. 1, 91--116. 


\bibitem{ABR}
S.~Alt{\i}nok, G.~Brown, M.~Reid,
\emph{Fano $3$-folds, K3 surfaces and graded rings.}
Topology and geometry: commemorating SISTAG, 25--53, Contemp. Math., 314, Amer. Math. Soc., Providence, RI, 2002.

\bibitem{BierstoneMilmanMinimal1}
E.~Bierstone, P.~Milman,
\emph{Resolution except for minimal singularities I}, 
Adv. Math. 231 (2012), no. 5, 3022--3053. 


\bibitem{Iano-Fletcher}
A.~Iano-Fletcher, 
\emph{Working with weighted complete intersections}, 
Explicit birational geometry of 3-folds, 101--173, 
London Math. Soc. Lecture Note Ser., 281, Cambridge Univ. Press, Cambridge, 2000. 



\bibitem{Fujino}
O.~Fujino, 
\emph{Multiplication maps and vanishing theorems for toric varieties}, 
Math. Z. 257 (2007), no. 3, 631--641. 


\bibitem{GLSintro}
G.-M.~Greuel, C.~Lossen, E.~Shustin,   
\emph{Introduction to singularities and deformations},
 Springer Monographs in Mathematics. Springer, Berlin, 2007. xii+471 pp.



\bibitem{Hayakawa2}
T.~Hayakawa, 
\emph{Blowing ups of 3-dimensional terminal singularities. II},
 Publ. Res. Inst. Math. Sci. 36 (2000), no. 3, 423--456.

\bibitem{Iskovskikh1}
V.~Iskovskikh
\emph{Fano threefolds. I},  
Izv. Akad. Nauk SSSR Ser. Mat. 41 (1977), no. 3, 516--562, 717. 

\bibitem{Iskovskikh2}
V.~Iskovskih,
\emph{Fano threefolds. II},  
Izv. Akad. Nauk SSSR Ser. Mat. 42 (1978), no. 3, 506--549. 

\bibitem{IPencyclopedia}
V.~Iskovskikh, Y.~Prokhorov, 
\emph{Fano varieties. Algebraic geometry, V}, 1--247, 
Encyclopaedia Math. Sci., 47, Springer, Berlin, 1999. 

\bibitem{kawamatacrelle}
 Y.~Kawamata,
 \emph{Minimal models and the Kodaira dimension of algebraic fiber spaces},
  J. Reine Angew. Math. 363 (1985), 1--46.

\bibitem{kollarmoriflip}
J.~Koll\'{a}r, S.~Mori, 
\emph{Classification of three-dimensional flips.},  
 J. Amer. Math. Soc. 5 (1992), no. 3, 533--703. 

\bibitem{KSB}
J.~Koll\'{a}r, N.~Shepherd-Barron, 
\emph{Threefolds and deformations of surface singularities}, 
Invent. Math. 91 (1988), no. 2, 299--338.

\bibitem{mina}
T.~Minagawa, 
\emph{Deformations of $\mathbb{Q}$-Calabi-Yau $3$-folds and $\mathbb{Q}$-Fano $3$-folds of Fano index $1$}, 
J. Math. Sci. Univ. Tokyo 6 (1999), no. 2, 397--414. 



\bibitem{mori}
S.~Mori, 
\emph{On $3$-dimensional terminal singularities}, 
Nagoya Math. J. 98 (1985), 43--66. 


\bibitem{mukai}
S.~Mukai, 
\emph{New developments in the theory of Fano threefolds: vector bundle method and moduli problems}, 
Sugaku Expositions 15 (2002), no. 2, 125--150. 

\bibitem{Namtop}
Y.~Namikawa, 
\emph{On deformations of Calabi-Yau $3$-folds with terminal singularities}, 
 Topology 33 (1994), no. 3, 429--446.

\bibitem{NamFano}
Y.~Namikawa, 
\emph{Smoothing Fano $3$-folds}, 
J. Algebraic Geom. 6 (1997), no. 2, 307--324. 

\bibitem{NamSt}
Y.~Namikawa, J.~Steenbrink,
\emph{Global smoothing of Calabi-Yau threefolds}, 
Invent. Math. 122 (1995), no. 2, 403--419. 


\bibitem{ReidKawamata} 
M.~Reid, 
\emph{Projective morphisms according to Kawamata}, Warwick preprint, 1983,
www.maths.warwick.ac.uk/~miles/3folds/Ka.pdf


\bibitem{YPG}
M.~Reid, 
\emph{Young person's guide to canonical singularities}, 
 Algebraic geometry, Bowdoin, 1985 (Brunswick, Maine, 1985), 345--414, 
Proc. Sympos. Pure Math., 46, Part 1, Amer. Math. Soc., Providence, RI, 1987.



\bibitem{Sano}
T.~Sano, 
\emph{On deformations of Fano threefolds with terminal singularities}, 
arXiv:1203.6323v5. 

\bibitem{Sanonote}
T.~Sano, 
\emph{On deformations of Fano threefolds with terminal singularities II}, arXiv: 1403.0212. 


\bibitem{Sernesi}
E.~Sernesi, 
\emph{Deformations of algebraic schemes}, 
Grundlehren der Mathematischen Wissenschaften,  
334. Springer-Verlag, Berlin, 2006. xii+339 pp.



\bibitem{Shokurov}
V.~Shokurov, 
\emph{Smoothness of a general anticanonical divisor on a Fano variety}, 
 Izv. Akad. Nauk SSSR Ser. Mat. 43 (1979), no. 2, 430--441.




\bibitem{TakagiDuVal}
H.~Takagi, 
\emph{Classification of primary $\mathbb{Q}$-Fano threefolds with anti-canonical Du Val K3 surfaces. I},
 J. Algebraic Geom. 15 (2006), no. 1, 31--85.

\bibitem{takagi2}
H.~Takagi, 
\emph{On classification of $\mathbb{Q}$-Fano $3$-folds of Gorenstein index $2$. II},  
Nagoya Math. J. 167 (2002), 157--216.


\bibitem{wahl}
J.~Wahl, 
\emph{Equisingular deformations of normal surface singularities. I}, 
Ann. of Math. (2) 104 (1976), no. 2, 325--356. 

\bibitem{wahlpersonal}
J.~Wahl, \emph{Personal communication}. 

\end{thebibliography}
\end{document}